\documentclass[conf]{new-aiaa}
\usepackage[utf8]{inputenc}
\usepackage{algorithm}
\usepackage{algpseudocode}
\usepackage{graphicx}
\usepackage{amsmath}
\usepackage[version=4]{mhchem}
\usepackage{siunitx}
\usepackage{longtable,tabularx}
\usepackage{multicol}
\usepackage{booktabs}
\usepackage{float}
\usepackage{lscape}
\setcitestyle{authoryear,open={(},close={)}}
\usepackage{xcolor}
\newtheorem{theorem}{Theorem}
\newenvironment{proof}{\paragraph{Proof:}}{\hfill$\square$}
\defcitealias{EPA2018}{EPA, 2018}

\title{Electric Vehicle Traveling Salesman Problem with Drone with Fixed-time-full-charge Policy}

\author{Tengkuo Zhu\footnote{Email: zhutengkuo@utexas.edu}}
 \affil{University of Texas at Austin, Texas.}
 \author{Stephen D. Boyles\footnote{Email: sboyles@mail.utexas.edu}}
 \affil{University of Texas at Austin, Texas}
 \author{Avinash Unnikrishnan\footnote{Email: uavinash@pdx.edu}}
 \affil{Portland State University, Portland, Oregon}

\begin{document}

\maketitle

\begin{abstract}
The idea of deploying electric vehicles and unmanned aerial vehicles (UAVs), also known as drones, to perform "last-mile" delivery in logistics operations has attracted increasing attention in the past few years. In this paper, we propose the electric vehicle traveling salesman problem with drone (EVTSPD), in which the electric vehicle (EV) and the drone perform delivery tasks coordinately while the electric vehicle may need to visit charging stations occasionally to recharge. We further assume that the EV can refresh its energy to full battery capacity with fixed time at charging stations. Thus, the proposed problem is termed EVTSPD-FF. In this paper, an arc-based mixed-integer programming model defined in a multigraph is presented for EVTSPD-FF. An exact branch-and-price (BP) algorithm and a variable neighborhood search heuristic are developed to solve instances with up to 25 customers in one minute. Numerical experiments show that the heuristic is much more efficient than solving the arc-based model using the ILOG CPLEX solver and BP algorithm. A real-world case study on the Austin network and the sensitivity analysis of different parameters are also conducted and presented. The results indicate that drone speed has a more significant effect on delivery time than the EV's driving range.  \par
\vspace{0.5cm} 

\end{abstract}

\section{Introduction}
Although e-commerce was growing fast before COVID-19 hit, the ongoing pandemic has accelerated the shift towards a more digital world and pushed more consumers worldwide online. Based on the United Nations conference on trade and development, the e-commerce sector saw a significant rise from 16 percent to 19 percent in 2020 in its share of all retail sales. According to the US commerce department, in the United States, online sales hit \$791.70 billion in 2020, up 32.4\% from \$598.02 billion in the prior year. With consumers increasingly turning to e-commerce for their shopping needs, efficient order fulfillment and parcel distribution is the expectation of every online shopping experience. As a result, businesses in this area have begun racing to develop new technologies and experimental supply chain models to increase parcel volume and expedite deliveries. \par

One of the most significant expenses and challenges in the item shipping process is same-day, last-mile delivery, where the goods are transported from a distribution hub to the customer's front door. The last mile of delivery cost accounts for about half of the total shipping costs, as the final leg of shipment typically involves multiple stops with low drop sizes, especially in urban areas. In recent years, a large variety of alternative last-mile transportation modes appeared in both academia and industry, such as electric vehicles \citep{lin2016electric}, autonomous robots \citep{jennings2019study} and unmanned aerial vehicles, or drones \citep{song2018persistent}. Some new delivery concepts are also proposed, such as crowdsourcing delivery \citep{rouges2014crowdsourcing}. The problem researched in this paper involves the operation of two new techniques mentioned above: electric vehicles and drones. \par

Meanwhile, political and practical considerations  in recent years, such as concern about greenhouse gas emissions, are driving the move away from internal combustion engines towards trucks and vans using alternative power sources. The deployment of a new generation of electric vehicles is critical for reducing greenhouse gas emissions. This concept has also been extended to logistics. The new invention of electric delivery trucks or vans is increasingly replacing the traditional delivery truck as an approach for logistics companies to "boost their bottom lines" while fighting climate change and urban pollution. UPS, Amazon, and DHL all announced their initiatives to partially or fully replace their current pickup and delivery fleet with battery-powered vehicles. \par

At the same time, auto manufacturers compete with start-ups to produce the most efficient and "smart" electric delivery vehicles. One of the latest inventions is small electric delivery vans, a picture of which is shown in Figure \ref{fig:small_ev_van}. It is an electric, lightweight van designed especially for small or medium-sized package delivery. Compared to the traditional bulky delivery truck, this small van is significantly more eco-friendly and cost-effective. This van is compact in dimensions (140-200 cm of width), can carry a load up to 1600-2000 kilograms, and has a maximum driving range of 60-200 km with a lithium battery. This small-sized and flexible electric vehicle might represent the future landscape of logistics.

\begin{figure}
    \centering
    \includegraphics[width = .4\textwidth]{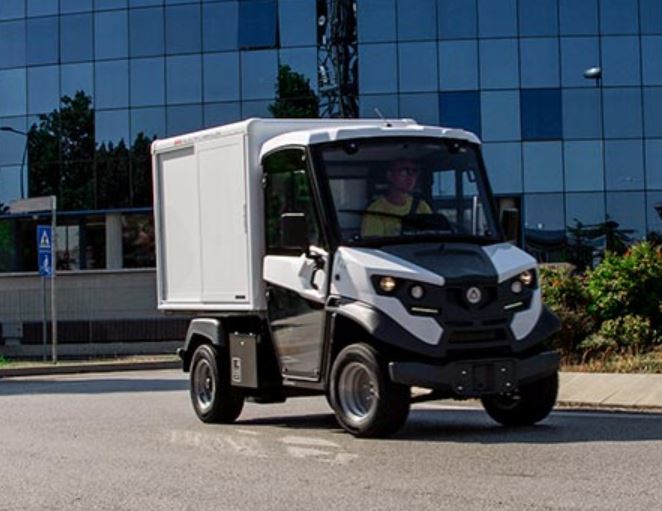}
    \caption{A small electric van created by Alkè, an Italy-based company} \label{fig:small_ev_van}
\end{figure}

Delivery via unmanned aerial vehicles, or drones, has also shown to have great potential to satisfy future customer expectations for delivery service quality as efficient, environmentally friendly, low maintenance cost delivery modes. The use of drones as a delivery method has been applied in several areas such as logistics, military operations, public security, traffic surveillance, monitoring, and humanitarian relief. In particular, the use of drones in logistics has sparked continuous interest for the general public ever since online retail giant Amazon announced its "Prime Air" initiative back in 2013. Multiple companies, such as Google, DHL, UPS, and Walmart, have revealed and tested drone delivery services in various countries. Compared to ground transportation delivery, drone delivery has the clear advantage of greater traveling speed, as it can travel directly from the launch point to the customer's location without being affected by the congested ground traffic. This property is crucial for several time-sensitive applications such as blood and medicine delivery. Besides, based on the results reported in \citep{goodchild2018delivery}, drones tend to have a $ CO_{2}$ emissions advantage over trucks in service zones that are either closer to the depot or have smaller numbers of customers. The past few years have witnessed a dramatic increase in UAV applications \citep{Dronezon}. According to Teal Group's prediction, commercial use of UAVs will grow eightfold over the decade to reach US \$7.3 billion in 2027 \citep{Teal}. For the rest of the paper, the term "UAV" and "drone" are used interchangeably.\par

However, these new techniques do not come without drawbacks. For electric vehicles, although their driving range has been dramatically improved recently due to the advancement in the power grid system and lithium battery, this value ranges typically from 150 miles to 300 miles. The driving range is significantly smaller than electric passenger cars for normal-sized electric trucks or vans for commercial use. For example, the expected driving range of electric delivery vans built by Rivian, an American electric vehicle automaker, is about 150 miles \citep{Insideevs}. The Mercedes-Benz eSprinter for Amazon has a maximum range of 104 miles with a 55 kWh battery and 70 miles with a 41 kWh battery \citep{alke}. As for small-sized electric vans, their driving range is about 60 miles to 190 miles with a 10-20 kWh lithium battery. Thus, for electric vans that can only cover less than 100 miles per charge, their success and fluent operation require them to recharge their battery en route when necessary. For drone delivery, apart from the limitation of related regulations of the Federal Aviation Administration and privacy issues in the urban area, its drawback mainly lies in a limited capacity and flight range. A rotary-winged drone, such as the one used in project "Amazon prime air," typically has a capacity of up to five pounds, with a maximum flight range of 15 miles. These values also indicate that a launched drone can only serve one customer before it is retrieved and recharged in a practical scenario. \par

Several new delivery concepts have been proposed in academia to mitigate the disadvantages of drone delivery, one of which is the in-tandem delivery method where the trucks are teamed up with drones to accomplish delivery tasks. Initially presented in \citep{Murray2015}, this idea requires that in operation, the truck carries a large number of parcels and charging facilities for the drones and serves as a drone hub while delivering parcels independently. The truck could launch the drone and retrieve the drone later before the drone's energy is depleted. This setting leads to a combinatorial optimization problem that requires the coordination of both truck and drone routes. \par

This paper describes a similar problem while replacing the traditional delivery truck with a small electric truck/van. This problem is called the electric vehicle traveling salesmen problem with drone (EVTSPD). A single truck/van equipped with a drone performs the delivery task in this problem. On its route, it needs to visit charging stations in the network to recharge its battery energy when necessary. Besides, we further assume that the truck/van shares its energy with its equipped drone. When the drone is launched from the truck/van to perform a delivery task, the energy required for drone delivery should be deducted from the remaining energy of the truck. It turns out that the addition of this assumption significantly increases the difficulty of the problem and needs specially constructed techniques to handle it. From a practical point of view, by combining two newly emerged techniques - EV and drone - we can gain some insight into the future landscape of logistics by extending current solution techniques to more complicated problem variants we may face later.

In general, the main contributions of this paper are:
\begin{itemize}
    \item The paper investigates the combination of EV and drone routing and propose the electric vehicle traveling salesman problem with drone. In EVTSPD, the proposed shared energy assumption ensures that the coordination of two vehicles should be considered throughout the operation. 
    \item In terms of the charging policy, this paper assumes that the EV could refresh its battery to full capacity with fixed amount of time every time it visits the charging station nodes. This proposed problem, EVTSPD with fixed-time-full-charge policy, is termed as EVTSPD-FF.
    \item A mixed integer/linear programming formulation of EVTSPD-FF is presented. This formulation is defined in a constructed multigraph, which contains no charging station nodes at the expense of an increased number of arcs. It is estimated that when three charging stations exist in the instance, the constructed multigraph has approximately 3-5 times more arcs compared to the original network.
    \item We present a set partitioning formulation for EVTSPD-FF and propose an exact branch-and-price (BP) solution method based on the multigraph. A dynamic programming (DP) approach is used to derive the least-cost route in the pricing problem of the branch-and-price algorithm, based on the ng-route relaxation of the problem. This approach proved to be efficient in identifying the columns which would enter the master problem of the BP algorithm.
    \item The numerical analysis results indicate that the BP method can solve instances containing up to 10 customers to proven optimality in one hour, which is significantly more efficient than solving the problem via a commercial solver.
    \item A variable neighborhood search (VNS) heuristic method is also proposed in this paper to solve EVTSPD-FF of practical size.
    \item Extensive computational runs are conducted to test the efficiency of the proposed MILP model, BP algorithm, and VNS heuristic. A case study based on a real-world Austin network is also performed, along with the sensitivity analysis of several critical parameters of EV and UAV.  
\end{itemize}

The rest of the paper is organized as follows. Section 2 provides a thorough literature review of the related and recent research on  EV routing and drone routing problems. Section 3 describes the EVTSPD-FF and its MILP formulation. Section 4 proposes the set partitioning formulation and the exact BP algorithm used to solve the problem. Specifically, it introduces the DP method used to solve the pricing problem of BP. Besides, a VNS method is also proposed in this section. Section 5 presents the computational results of the BP algorithm compared to MILP models, the performance evaluation of VNS, and a real-world case study. Finally, section 6 concludes the work and gives insight into potential future research directions.

\section{Literature Review}
The number of publications related to EVs or drones has grown significantly in recent years. A thorough and comprehensive on both areas are not the main subject of this section. Thus we only include recent papers that focus specifically on EV and drone routing. 

\subsection{EV-related routing}
The electric vehicle routing problem is a recent variant of traditional VRP. More details about VRP and its variants could be seen in \citep{braekers2016vehicle, montoya2015literature}. 
EVRP is a branch of the alternative-fuel powered vehicle routing problem where the conventional vehicles are replaced by alternative-fuel powered vehicles (AFVs), which have limited fuel tank capacity or battery. Thus, EVRP is also a branch of the green vehicle routing problem (GVRP). In fact, EVRP is called GVRP in the early papers of this topic, such as \citep{Erdogan2012}.  Readers can refer to \citep{asghari2020green} as the latest literature review on GVRP. For a review paper that specially focuses on the electric vehicle routing problem and its variants, readers can refer to \citep{erdelic2019survey}.\par

There are two basic configurations of EVs: the battery electric vehicle (BEV) and the hybrid electric vehicle (HEV). The former one is powered exclusively by batteries inside the vehicle, while the latter one can be powered by batteries and an internal combustion engine. Consequently, BEV has much less driving range compared to HEV. Similar to a majority of EV routing-related papers regarding BEV as EV, in this problem, we also assume that the electric van is BEV. For the rest of the paper, we assume that EV is referred to as BEV if not clearly stated otherwise. An EV could renew its battery energy at charging stations in two ways: by swapping empty batteries with fully charged ones or by plug-in charging via an electric power source. The battery swapping process could be finished within minutes, while the plug-in charging time depends on the desired state of charge level and charging function. \par

The early research on the routing of electric fleet includes \citep{gonccalves2011optimization, conrad2011recharging}. The former focuses on VRPPD with mixed feet of EVs and conventional internal combustion engine vehicles where the charging time is based on the total distance traveled, while the latter one assumes that vehicles are allowed to recharge at customer locations to recover up to 80\% of its charge capacity. The numerical analysis results of the latter paper imply that customer time windows have a greater impact on the tour distance recharging time is long. \citep{Erdogan2012} formally formulated GVRP in which a fleet of AFVs are involved. The AFVs can refresh their driving electricity at charging station nodes with fixed charging time. No customer time window and vehicle capacity are considered in the problem. The MILP formulation is proposed based on the augmented network. Two problem-specific heuristics are presented: one is based on the Clarke and Wright Savings algorithm, while the other one is based on the clustering algorithm. \citep{schneider2014electric} proposed a problem that considers both the service time window and vehicle capacity. This problem assumes a constant and linear charging rate at the charging station node, while the recharging time depends on the remaining energy level of the vehicle when arriving at charging station nodes. A hybrid variable
neighborhood search heuristic/tabu search method is used to solve the problem. In the previously discussed paper, the charging station nodes are copied multiple times to create the augmented network to enforce that each node could visit at most one time. A different approach is adopted in \citep{bruglieri2016new} where for each pair of customers, the charging stations are pre-computed to ensure that the vehicle can travel through the two customers via the computed charging stations. A new MILP model is proposed based on this network. However, the limitation of this approach is that only one charging station node can be traversed between two consecutive customer visits, which could be sub-optimal on some occasions. \par

All the above-mentioned papers use the battery swapping scheme and assume fixed charging time, which is not accurate in practice. In fact, the recharging of the battery employs a multi-phase charging scheme \citep{gao2002dynamic, tremblay2009experimental} to protect the battery from over-charging, which makes the function of charging time nonlinear. 
\citep{Montoya2017} extended the previous E-VRP models to consider nonlinear charging functions and used a piece-wise linear approximation to simplify the process. The numerical analysis in the paper indicates that neglecting nonlinear charging might lead to an infeasible route. A local search meta-heuristic is used to solve the problem.
\citep{froger2019improved} proposed two new formulations for nonlinear charging EVRP: one is the arc-based formulation, and another one is path-based formulation. The numerical analysis results show that the path-based model outperforms the one in \citep{Montoya2017}. \citep{kocc2019electric} extended the problem by considering shared charging stations. This new problem aims to minimize the total charging station construction cost and driving cost. A similar problem is proposed in \citep{kullman2021electric}, in which the author considers EVRP with public charging stations. In the public charging station, if the demand is full, the waiting time is modeled with queuing theory. A decomposition approach is used to derive the bounds of the problem. \par

With the complex constraints of EVRP, most of the papers use heuristics/meta-heuristics to obtain a solution. Exact methods for the EVRP and its variants are rather scarce. \citep{kocc2016green} presented a new formulation while assuming that no more than one refueling stop is made by any vehicle when traveling between two customers or the depot and one customer, which is similar to the approach used in \citep{bruglieri2016new}. A branch-and-cut algorithm is used to obtain the lower bounds, while a heuristic based on simulated annealing is adopted to obtain upper bounds. This solution method can solve instances containing up to 20 customers. \citep{desaulniers2016exact} adopts an exact branch-price-and-cut algorithm to solve four variants of EVRP. This algorithm uses customized monodirectional and bidirectional labeling algorithms for generating feasible vehicle routes and can solve instances with 100 customers and 21 recharging stations. \citep{hiermann2016electric} also reports an exact branch-and-price algorithm and adaptive large neighborhood search method for EVRP with time window and various fleet sizes. More recently, \citep{andelmin2017exact} proposed an exact branch-price-and-cut algorithm for solving EVRP with maximum duration constraints based on the creation of a multigraph. This approach avoids the limitation of the approach used in \citep{kocc2016green} and is shown to have excellent performance, which can solve instances with 110 customers to optimality. This multigraph approach is also used in this paper. \citep{Lee2021} presents an exact branch-and-price approach for the EVRP with nonlinear charging functions by introducing an extended charging stations network. This extended charging stations network is very similar to the refueling path network in this paper. However, in its branch-and-price scheme, the subproblem only obtains route segment and not complete routes, which is different from our approach.

\subsection{Drone-related routing}
The number of recent publications related to the application of drones is growing extremely fast. \citep{otto2018optimization} review several important and potential civil drone applications in the area of agriculture, monitoring, transport, security, etc. This section does not include a detailed review of the drone-related application, so we only consider the papers that involve drone routing-related optimization problems. A recent review on this subject can be seen in \citep{macrina2020drone}.\par

Several works emerged in recent years that focus on serving customers with a fleet that is composed of drones only. Most of these drone-only routing problems pay special attention to the energy consumption, limited capacity, and flight range of drones. Unlike the truck-drone in-tandem routing problem where a single drone is assumed to serve only one customer per trip, in several drone-only problems, it is normally assumed to be capable of serving multiple customers in one trip. \citep{dorling2016vehicle} proposes two models of drone-only VRP with different objectives.  The first model minimizes costs under the delivery time limit, while the second one minimizes the delivery time under budget constraints. A simulated annealing (SA) heuristic is used to solve the problems. \citep{coelho2017multi} introduced a new problem named the green UAV routing problem, where multiple charging station nodes are introduced. This model also has multiple objective functions, which include total travel distance/time, the number of drones used, and so on. A heuristic called the multi-objective smart pool search (MOSPOOLS) is used to obtain the solution. \citep{liu2019optimization} present a model where drones are used to perform on-demand meal delivery services. It is a pick-up and delivery problem where the meal orders have to be fulfilled in a dynamic order. A time-discretized MIP formulation and a heuristic are proposed in the paper. Additional researches related to this topic includes \citep{9157260} and \citep{choudhury2021efficient}.\par

Another class of drone routing problem is the one that focuses on truck-drone in-tandem delivery systems. One obvious advantage of this delivery method is the drone can greatly extend their service radius because now the truck could be served as a drone hub that launch/retrieve the drone. Besides, it also can be seen as a natural transitional method from the traditional truck-only delivery to fully automated drone fleet delivery. This newly emerged delivery method is firstly proposed in \citep{Murray2015}, where it introduces the flying sidekick TSP (FSTSP), where a single truck and a single drone aim to serve a set of customers. The truck and drone can perform delivery tasks independently, while the drone could be launched from the truck and retrieved later at a different location. The drone can only carry one parcel per trip, and constant launching/retrieving time is considered. The problem aims to minimize the coordinated route duration, including the waiting time caused by the synchronization of both vehicles. A MILP formulation, along with two heuristics, is also proposed in the paper.  Since then, various research has been conducted, either trying to improve the solving efficiency of FSTSP or to explore new ways of combining multiple trucks and drones. \par

\citep{ferrandez2016optimization, Agatz2018, bouman2018dynamic, EsYurek2018, DeFreitas2018, Ha2018} are some of the examples that aim to study the variant of FSTSP or to propose new solution methods. In particular, \citep{ferrandez2016optimization} study the effectiveness of this new delivery system and adopts K-means algorithms to determine optimal launch locations and a genetic algorithm to obtain truck route between launch locations. It finds that the new system is more efficient in terms of energy and time. Similar analyses are also performed in \citep{carlsson2018coordinated} where it assumes that drones can be launched/retrieved at some points which are not necessarily customer locations. \citep{Agatz2018} study a variant of FSTSP where the drone can be launched/retrieved at the same customer location. This variant is named the traveling salesman problem with drones (TSPD). Two route-first, cluster-second heuristics are proposed to solve TSPD. \citep{bouman2018dynamic} present an exact dynamic programming algorithm to solve the same problem proposed in \citep{Agatz2018}. \citep{EsYurek2018} presents an iterative method to solve FSTSP by pre-deciding the customers served by the truck or the drone, where for each case, the truck route and drone route are obtained in sequence. This method performs well when the number of customers is small but unable to solve problems containing more than ten customers. \citep{DeFreitas2018} propose a variable neighborhood search method to solve both FSTSP and TSPD. \citep{Ha2018} study a variant of FSTSP where the objective function aims at minimizing total operational costs, which includes transportation cost and waiting cost. A heuristic method based on a greedy randomized adaptive search procedure (GRASP) is used to solve the problem. More recently, \citep{salama2020joint} studied a variant of FSTSP where multiple drones operate in conjunction with a single truck. In this problem, the truck needs to launch several drones at a focal point in each customer cluster. A three-step heuristic is presented, where the first step is to find the focal point and then assign customers to focal points, and lastly, the routing of a truck among these focal points.  Besides, \citep{RajMurray2020mfstsp} also studies FSTSP with multiple drones and considers different travel speeds, payload capacities, service times, and flight range. A MILP formulation is proposed, and a heuristic approach consisting of three subproblems is proposed. 

As a generalization of FSTSP and TSPD, the vehicle routing problem with drones (VRPD) is another attempt to combine trucks and drones in a more complex way. \citep{Wang2017} introduced VRPD, where a fleet of trucks, each equipped with multiple drones, aims to perform the delivery task. In this problem, drones can be launched/retrieved at any of the customer locations. An analysis of several worst-case scenarios is conducted, from which the best possible time savings can be derived. \citep{schermer2018algorithms} propose two route-first cluster-second heuristics to solve VRPD, where the improving phase is based on local search.  \citep{wang2019vehicle} extend the problem by proposing an arc-based MILP formulation. In this problem, there exist several drone hub nodes that serve as a hub that can refresh/reload drones. A branch-and-price algorithm is used to solve this problem.  \citep{schermer2019hybrid} considers an extension of VRPD, called vehicle routing problem with drones and en route operations (VRPDERO). In VRPDERO, drone operations might also start and end at some discrete points on the arcs traversed by the vehicles, and the objective is to minimize the makespan. A MILP formulation is proposed, and a variable neighborhood search is used to obtain a solution to the problem. \citep{sacramento2019adaptive} presents an adaptive large neighborhood search method to solve a VRPD variant where each truck is equipped with exactly one drone. \citep{kitjacharoenchai2020two} considers a two-echelon VRPD where truck routing and drone routing are separate. A similar problem is also seen in \citep{liu2020two}.

In this paper, we propose and analyze the the routing problem that involves both the EV and the drone, which is named as EVTSPD. In this problem, the shared energy assumption indicates that the synchronization between the EV and the drone has to be considered. Specifically, we focus on EVTSPD-FF which further assumes that the EV could refresh its battery energy with fixed amount of time every time it visits the charging station nodes. This assumption allows us to redefine the problem in a constructed multigraph. A MILP formulation is proposed based on the idea of multigraph, along with an exact branch-and-price scheme and an efficient variable neighbourhood search algorithm to solve the problem of different size.

\section{Problem Description and Mathematical Formulation of EVTSPD-FF}
In EVTSPD, a central planner aims to serve all the customers in the network using a single EV equipped with a single drone, which is initially parked at the depot. In the operation process, the EV and the drone perform delivery tasks in tandem, where the EV can serve as the drone hub that can launch, retrieve and reload the drone, and both the EV and the drone can serve customers independently. This operation pattern is similar to the flying sidekick traveling salesman problem introduced in \citep{Murray2015} and the traveling salesman problem with drone studied in \citep{Agatz2018}. However, the most significant difference between EVTSPD and the previous TSPD research is that the EV is associated with driving energy which decreases gradually as it traverses arcs in the network. Thus, in EVTSPD, the EV needs to visit specialized charging station nodes in the network to recharge the energy. The objective of this problem is to obtain an EV-drone coordinated route that has minimal operational time. \par

In the modeling process, the following assumptions or simplifications are adopted, most of which are in analogy to existing VRPs and EVRPs:

\begin{enumerate}
    \item The drone is assumed to be a rotary-wing UAV so that it can take-off and land vertically on the EV.
    \item Both the EV's and the drone's speed are assumed to be constant, and the altitude differences in the network are ignored.
    \item The consumed energy of traveling is assumed to be a linear function of traveled distance. 
    \item At the beginning of the planning horizon, both the EV and the drone are located at the depot and the EV is assumed to be fully-charged.
    \item The EV can recharge its energy at the charging stations in the network. The charging process is assumed to start immediately as the EV arrives at the station. It is also assumed that all charging stations have unlimited capacities.
    \item All charging stations are considered to be level-3 fast charging stations so that it would not lead to unacceptable high charging times.
    
\end{enumerate}

The EVTSPD is defined on an undirected, complete graph $G = (V,E)$, with a node set $V$ consisting of the customer set $C = \{1, 2, . . . ,c\}$, the depot set $\{0, 0^{'}\}$, and a set of charging stations (CS) nodes $S = \{c+1,c+2, . . . ,c+s\}$. The node set is thus $V = \{0, 0^{'}\} \cup C \cup S$ and $|V| = c + s + 1$. . The edge set $E = \{(i,j): i,j \in V, i < j\}$ contains the edges connecting vertices of $V$. Each edge $(i,j)$ is associated with two non-negative travel time cost $c^{T}_{ij}$ and $c^{D}_{ij}$, which corresponds to the travel time needed for the truck and UAV to travel from node $i$ to node $j$, respectively. In general, as UAV travels faster than the truck, $c^{T}_{ij}$ is normally no less than $c^{D}_{ij}$. Besides, each edge is also associated with two non-negative energy cost $e^{T}_{ij}$ and $e^{D}_{ij}$ for the truck and the UAV, respectively. Note that both $e^{T}_{ij}$ and $e^{D}_{ij}$ are also measured in time unit. To simplify the problem, we further assume that for all the edges $e^{T}_{ij} \leq Q^{T}$ and $e^{D}_{ij} \leq Q^{D}$, where $Q^{T}$ and $Q^{D}$ are energy capacity for the truck and drone, respectively. \par

In this paper, as we assume that the truck is an electric vehicle. The term "truck", "electric vehicle", or "electric truck" are used interchangeably unless with further instruction. Similarly, the term "UAV" and "drone" are also used interchangeably. The planner aims to serve all the customers in the network using the truck or the drone. In the operational stage, after they depart from the depot, both the truck and the drone can serve customers independently. Besides, the truck serves as the drone hub to launch/retrieve the drone at a customer node. After launching the drone, the truck continues its planned route and serves customers before retrieving the drone at another customer node. Meanwhile, the drone will also perform delivery tasks, after which it will be retrieved by the truck. We assume that the launching/retrieving time is zero (this assumption will be relaxed later in this section). In practice, this can be done if the truck carries multiple drones and only launches the fully loaded/charged ones after retrieving the drone in operation. We define the process of the drone's operation from the launch to the next retrieve as "drone sortie" or "UAV sortie." In each drone sortie, we assume only one customer can be served, considering the low loading capacity of the drone. This indicates that each drone sortie contains exactly three nodes: launch node, customer node (also called "drone node"), and retrieve node. Besides, waiting time might be generated for each drone sortie when the truck and the drone arrive at the retrieve node at different times. \par

The goal of EVTSPD is to find a coordinated tour that starts and ends at the depot and visits a subset of vertices (including charging stations when necessary) such that the total delivery time is minimized. All customers must be served once, either by the EV or by the drone. 

The final feasible solution of EVTSPD consists of the EV route and several non-overlapping UAV routes, which start and end on the EV route. A simple example is shown in Figure \ref{fig:EV-UAV coordinated route}, where the EV route is $\{0-1-3-5-6-0 \}$ while the UAV route is $\{1-2-3, 3-4-6\}$. We define "drone node" as the node that is solely visited by the drone, "truck node" as the node that is solely visited by the truck, and "combined node" as the node that is visited by both the truck and the drone.

\begin{figure}
    \centering
    \includegraphics[width = .8\textwidth]{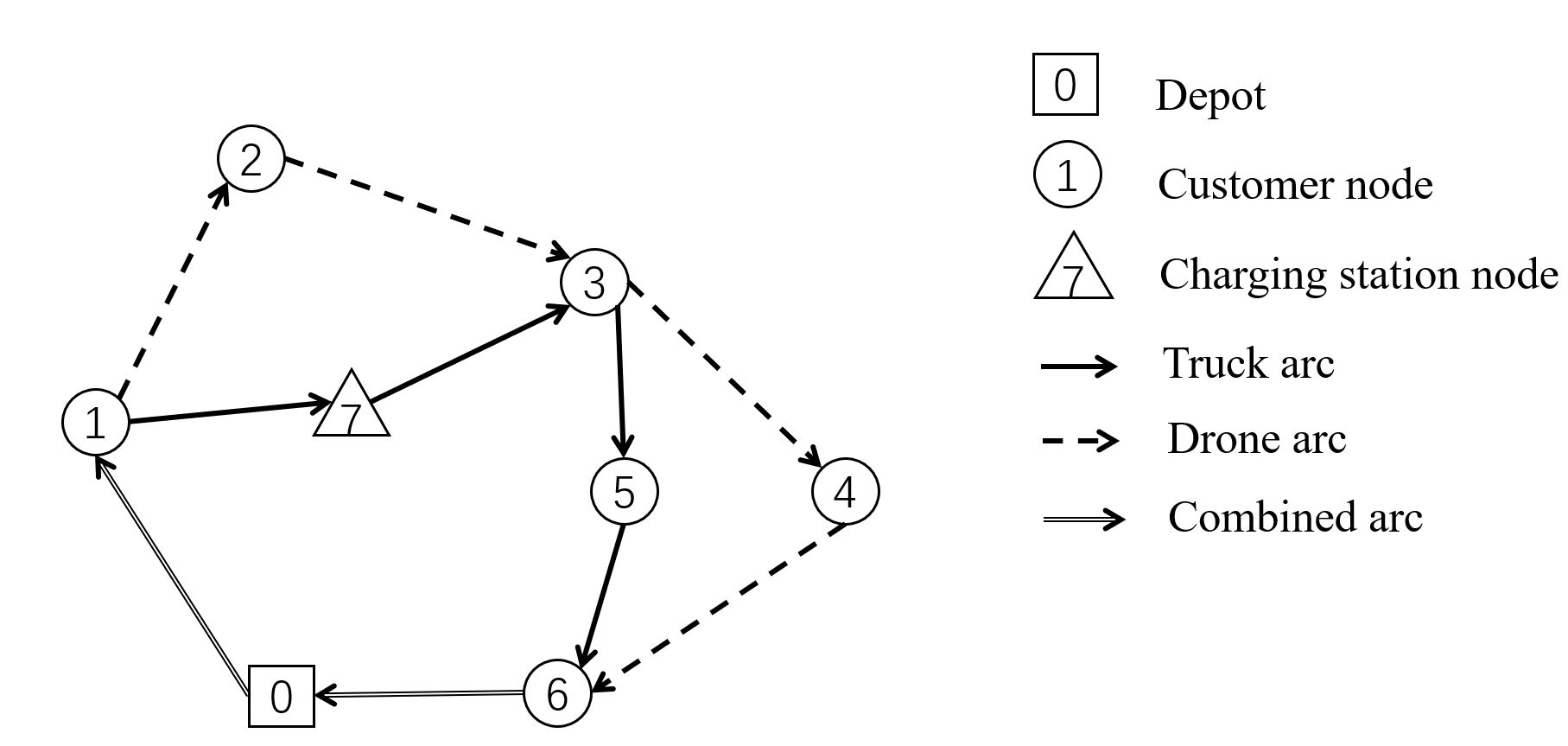}
    \caption{A simple representation of the EV-UAV coordinated route}
    \label{fig:EV-UAV coordinated route}
\end{figure}

Below is a summary of parameters that are available in the original network $G$:
{\renewcommand\arraystretch{1.0}
\noindent\begin{longtable*}
{@{}l @{\quad:\quad} p{15cm}@{}}
$c^{T}_{ij}$ & Travel time for the EV when traversing arc $(i,j)$\\
$c^{D}_{ij}$ &  Travel time for the drone when traversing arc $(i,j)$ \\
$d_{ijk}$ &  Travel time cost for the drone to launch from node $i$, serves node $j$ and return to node $k$\\
$e^{T}_{ij}$ & Driving energy cost for the EV when traversing arc $(i,j)$\\
$e^{D}_{ij}$ & Flight energy cost for the drone when traversing arc $(i,j)$\\
$e^{D}_{ijk}$ &  Flight energy cost for the drone to launch from node $i$, serves node $j$ and return to node $k$\\
$Q^{D}$  & Flight energy capacity of the drone \\
$Q^{T}$  & Driving energy capacity of the EV \\ 
$M$ & A positive large number\\
\end{longtable*}}

\subsection{The shared energy assumption}

One special assumption adopted in this research, distinct from previous VRP or TSPD research, is the shared energy assumption between the EV and the drone. To be more specific, we assume that there is a Lithium-Ion battery of limited capacity on the EV that can be used by both EV and drone. This battery not only serves as the energy source for the EV, it can also recharge the battery on the drone if it is on board. During the operation, if the drone is launched to serve a customer independently, the energy consumed by this drone sortie (which can be calculated beforehand) should be deducted from the remaining energy level of the battery. \par
To better illustrate this assumption, denote $b^{a}_{i}$ as the remaining energy of EV upon arrival at node $i$ and $b^{d}_{i}$ as the remaining energy of EV upon departure from node $i$.  Using the network in Figure \ref{fig:energy assumption}, assume the electric vehicle launches the UAV at node $1$, travels to node $3$, and retrieves the UAV at node $4$ while the UAV serves the customer $2$. If $b^{a}_{1} = 100$, which indicates that EV's remaining battery level upon arrival at node 1 is 100, then the EV's battery level upon departure from node 1 could be calculated as $b^{d}_{1} = 100 - e^{D}_{124} = 100-20 = 80$, where $e^{D}_{124}$ represents the required energy of the UAV to be launched at node 1, serves customer 2 and returns to node 4. The battery level of the other nodes is also shown in Figure \ref{fig:energy assumption}. \par
There are two major reasons to make this assumption. First of all, unlike bigger ones used in military applications, the smaller drones, mainly used in logistics, usually have an inadequate power supply. An EV-drone coordinated delivery system can solve this problem as the battery on EV can also be used by the drone. \par

Second, simply ignoring the consumed energy during the drone's operation would lead to inaccurate EV driving range and overall result, as the drone's energy consumption rate is comparable to that of an EV. The average energy consumption rate for EVs on the market is around 200 wh/km \citep{EVD}. For the specialized electric delivery van, whose driving range is about 86 km with 10 kWh LiFePO4 battery and 200 km with 20 kWh battery, the energy consumption rate is about 100-116 wh/km \citep{AlkeVan}. In terms of commercial drones, for example, DJI's MATRICE 600 PRO, when equipped with six TB47S batteries (4500 mAh, 22.2 V), can fly at most 16 minutes with a 6 kg payload, according to its official website \citep{DJI}. Assuming this situation, it can fly at a maximum speed of about 60 km/h. The drone can travel up to 16 km before the batteries are depleted, which indicates that the energy consumption rate is about 37.5 wh/km, about 20\% of a typical EV and 40\% of a smaller electric van.  The shared energy assumption aims to account for this extra 20\%-40\% percent of energy consumption from the drone.  Besides, to consider the difference between EVs and drones in their energy consumption rate, we introduce another parameter $\alpha$ to represent the energy consumption ratio between the drone and the EV.  In the numerical analysis section, this value is chosen as $\alpha = 0.4$. \par
In terms of the drone's charging time, in reality, it takes at least 30 minutes for the zero-energy-level battery to be fully charged. In practice, this time could be bypassed by swapping the used battery with the alternative backup batteries, which could be charged regardless of whether the drone is on-board or not. \par 

\begin{figure}
    \centering
    \includegraphics[width = .8\textwidth]{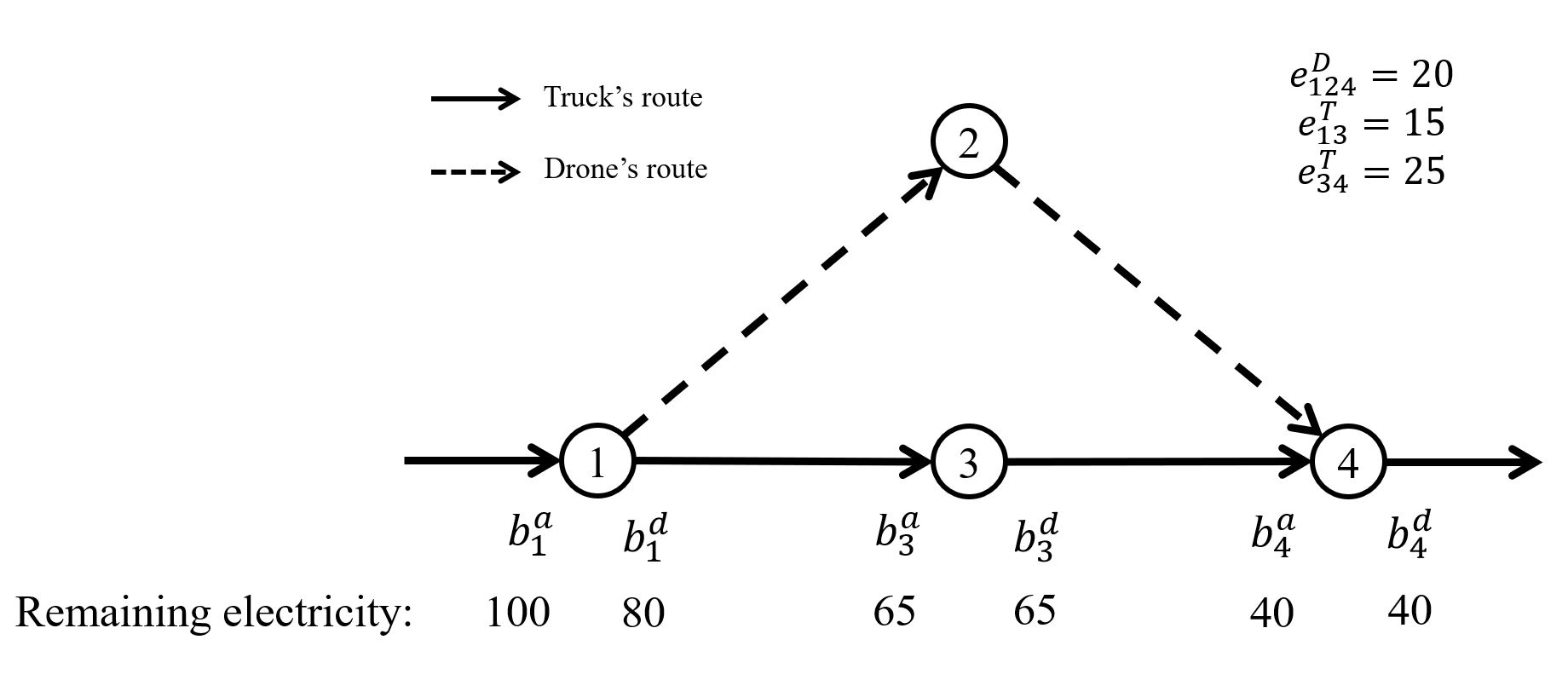}
    \caption{A simple representation of the shared energy assumption}
    \label{fig:energy assumption}
\end{figure}

\subsection{The charging policy}
In EVTSPD, the truck is assumed to be an electric vehicle associated with driving energy. As the truck traverses the arcs in the network, its remaining energy decreases gradually. Thus, the truck needs to visit charging station nodes when necessary to avoid its remaining energy being depleted. Note that no limit is set on the number of stops made for recharging at any charging stations. Thus, the final solution might contain more than one visit to some specific charging station nodes, while some other charging station nodes might never be visited at all.

In this paper, we assume that the electric vehicle is BEV and the charging stations are battery swapping stations (BSSs). For this setting, the BEV could renew its battery energy at battery swapping stations (BSSs) by swapping empty batteries with fully charged ones, which could be finished within minutes \citep{sarker2014optimal}. Thus, the EV is expected to refresh its battery to full capacity with a fixed time every time it visits a charging station. This fixed-time-full-charge policy for EV is proposed in the early EV-related research and is commonly adopted in the literature \citep{conrad2011recharging, Erdogan2012}. EVTSPD which adopts this assumption is named as EVTSPD-FF.

Although this fixed-time-full-charge assumption might not be widely adopted in practice, this paper later shows that this setting enables us to create a multigraph that only contains customer nodes. The EVTSPD-FF can be defined on this multigraph instead of the original network. In other words, this assumption enables us to achieve significant computational efficiency that is otherwise impossible if this assumption is discarded. Later in this paper, we will show how to exploit this property to solve the problem exactly with a branch-and-price algorithm. \par

\subsection{Construction of the multigraph}
Before we present our MILP model for EVTSPD-FF, one thing to note is that this proposed MILP formulation is not defined on the original network $G$ but a constructed multigraph $\mathscr{G}$. We will explain the motivations and reasons for constructing this new network before the detailed steps of the multigraph construction. \par

Compared to the normal TSPD, the major complication in EVTSPD-FF is the extra EV driving range constraints: the planner needs to decide when and which charging station nodes should be visited before the driving energy is depleted. However, generating such a feasible route based on the original network $G$ is difficult. To enable multiple visits to the same charging station nodes, one needs to create several "dummy" copies of the charging station node set and insert these "dummy" nodes back into $G$ to construct an augmented network $G_{0}$. In \citep{andelmin2017exact}, the authors propose an efficient way to overcome this difficulty: reformulate the original network $G$ into a multigraph. In this research, we also adopt this technique to construct a multigraph for two purposes: it serves as an alternative network that EVTSPD-FF could be defined on. The multigraph is also used to generate feasible routes in the BP algorithm introduced later. For the rest of this subsection, we will explain the steps of constructing this multigraph.

Denote $G_{S} = (S, E_{S})$ as the subgraph containing the charging station node set $S$, and arc set $E_{S}$. This arc set $E_{S}$ contains all the arcs between the CS nodes, i.e., $E_{S} = \{(k,l): k,l\in S, e^{T}_{kl}\leq Q^{T}\}$. Thus, network $G_{S}$ is a graph that only contains CS nodes and CS arcs. Let $C_{0} = C \cup \{0\}$ be the set containing all customers and the single depot. Meanwhile, we define \textit{refuel path} as a path that starts at $i \in C_{0}$, ends at $j \in C_{0}, i \neq j$ and traverses a path in graph $G_{S}$. Besides, if a vehicle travels from node $i$ to node $j$ directly by traversing arc $(i,j)$, the path is called a "direct path". Therefore, a path between two customer nodes is either a direct path or a refuel path.  For every direct path we assign a index 0 to it so that a direct path between $i$ and $j$ is represented as $(i,j,0)$. We also assign a unique index $p>0$ to every refuel path so that a refuel path between $i$ and $j$ is represented as $(i,j,p)$. \par

We can now define an extended arc set $\mathscr{E}$ that contains all the refuel paths and direct arcs $(i, j), \forall i, j \in C_{0}$. The refuel path $p$ that starts at $i$ and ends at $j$ is represented as arc $(i, j, p)$ in $\mathscr{E}$ and the direct arc that starts at $i$ and ends at $j$ is represented as arc $(i, j, 0)$ in $\mathscr{E}$. In this way, the arc set $\mathscr{E}$ contains all the paths that starts at a node $i \in C_{0}$, end at another node $j \in C_{0}$ and visits no other customers along the route, regardless how many CS nodes are visited between $i$ and $j$. \par

At this point, we can define a multigraph $\mathscr{G} = (C_{0}, \mathscr{E})$ where node-set $C_{0} = C \cup \{0\}$ is the set containing all customers and the single depot and $\mathscr{E}$ is the extended arc set. Each path from $i$ to $j$ in original network $G$ will be represented as an arc from $i$ to $j$ in the multigraph $\mathscr{G}$. Compared to the original network $G$, in the multigraph $\mathscr{G}$, the charging station nodes can be eliminated as they are inherently represented by each arc in $\mathscr{G}$. In this way, any original route in the original network $G$ can be represented as an alternative route in the multigraph $\mathscr{G}$, which has the same property as the original route.\par

Keep in mind that the main reason for constructing this multigraph $\mathscr{G}$ is to create an efficient alternative network to avoid the construction of the augmented network, where the charging station nodes are duplicated. Also, note that as the multigraph $\mathscr{G}$ contains all the direct arcs between nodes in $C_{0}$ in the original network $G$ and we prohibit the drones from being launched or retrieved at the charging station nodes, any feasible drone sorties in the original network $G$ corresponds to the same drone sortie in $\mathscr{G}$. For this reason, both the EV's route and drone's sorties can be constructed on the multigraph $\mathscr{G}$. This indicates that EVTSPD-FF can be modeled on multigraph $\mathscr{G}$. \par 

Next, we will introduce the properties associated with each arc in $\mathscr{E}$. Each arc $(i, j, p), p > 0$ that represents a refuel path $p$ is associated with three properties: minimum energy requirement $e^{T}_{(i,j,p)}$, remaining energy $e^{r}_{(i,j,p)}$, and travel time cost $c^{T}_{(i,j,p)}$. Minimum energy requirement of arc $(i,j,p)$ is the minimum energy required at node $i$ to traverse this arc. Remaining energy of arc $(i,j,p)$ indicates the remaining energy of the truck after it traverses this arc. Travel time cost of arc $(i,j,p)$ is the time cost of EV when traversing this arc. To calculate these properties, denote $e^{1}_{(i,j,p)}$ and $e^{2}_{(i,j,p)}$ as the energy consumption of the first and the last arc in original network $G$ traversed by refuel path $p$. Apparently, $e^{T}_{(i,j,p)} = e^{1}_{(i,j,p)}$ and $e^{r}_{(i,j,p)} = Q^{T} - e^{2}_{(i,j,p)}$. As for $c^{T}_{(i,j,p)}$, it is simply the sum of the travel time of all the original direct arcs in $p$. \par

As for each arc $(i, j, 0)$ that represents a direct arc $(i,j)$ in original network $G$, it only has two properties: minimum energy requirement and travel time cost, which are calculated in the same way as described above. Keep in mind that as arc $(i, j, 0) \in \mathscr{G}$ corresponds to a direct arc $(i, j)$ in $G$, its first and last arc are the same. Also note that this arc has no associated remaining energy $e^{r}_{(i,j,p)}$. This is because the truck visits no CS nodes when traversing direct arc. Thus, the remaining energy of EV after traversing direct arc $(i, j, 0)$ is not a fixed value, and thus no remaining energy is associated with this arc. \par

To better explain these properties in $\mathscr{E}$ and $E$, a simple case is shown in Figure \ref{fig:multi_illu}. In this figure, the first column represents the refuel path or direct arc in the original network $G$, and the second column shows their corresponding arcs in multigraph $\mathscr{G}$.

\begin{figure}[H]
    \centering
    \includegraphics[width = .9\textwidth]{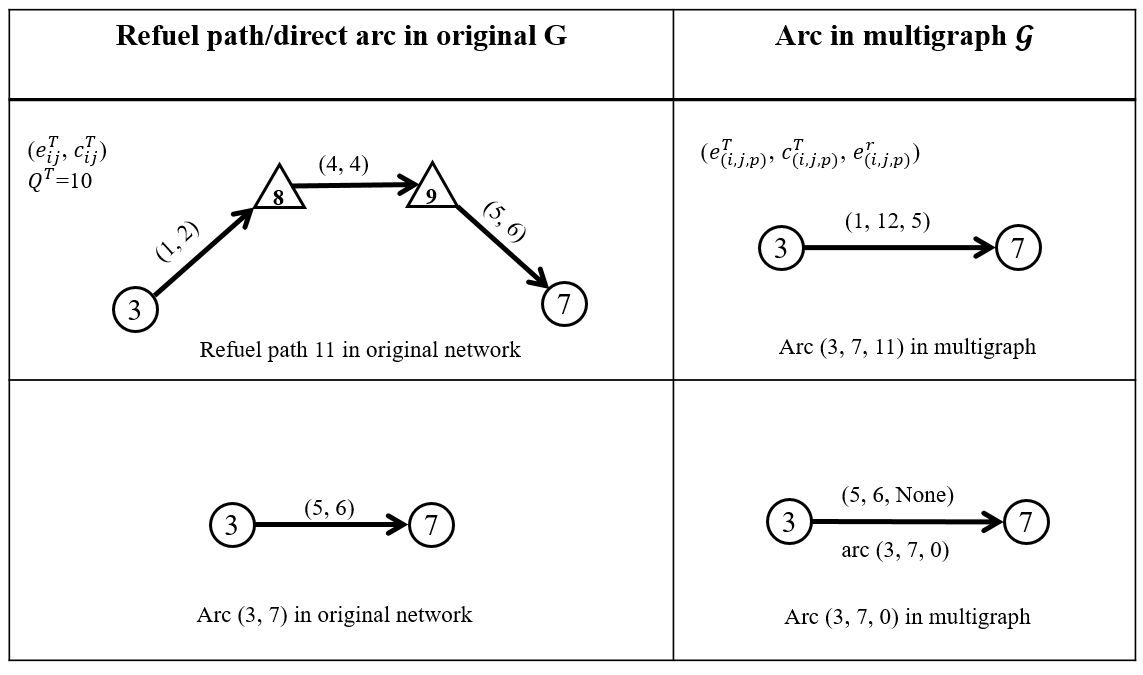}
    \caption{A representation of the refuel path or direct arcs in two networks}
    \label{fig:multi_illu}
\end{figure}

Theoretically, the number of arcs in $\mathscr{G}$ is prohibitive as the number of refuel paths between any two nodes is potentially infinite. However, a lot of these refuel paths can be eliminated from $\mathscr{G}$ as they are dominated by other paths and will not appear in the optimal solution of the problem. Below we introduce the dominance rule among the refuel paths.

\begin{theorem}
Let $P= (i,k,..., l,j)$ be a refuel path from $i \in N_{d}$ to $j \in N_{d}$. This refuel path is represented as arc $(i,j, p)$ in $\mathscr{G}$. Let $P^{'}= (i,k^{'},..., l^{'},j)$ be another refuel path from $i$ to $j$. This refuel path is represented as arc $(i,j, p^{'})$ in $\mathscr{G}$. Refuel path $P$ is said to be dominated by $P^{'}$ if the following conditions are satisfied: (i) $e^{T}_{(i,j,p)} \geq e^{T}_{(i,j,p^{'})}$, (ii) $e^{r}_{(i,j,p)} \leq e^{r}_{(i,j,p^{'})}$, (iii) $c^{T}_{(i,j,p)} \geq c^{T}_{(i,j,p^{'})}$ and at least one of the three inequalities is strict.
\end{theorem}

The above dominance rule is easy to understand and prove, which indicates that a path $P$ is better than (or "dominates") another one if it has less energy requirement, greater remaining energy, and less travel time. As there exists at least one optimal solution to EVTSPD-FF whose traversed refuel paths are non-dominated, only non-dominated refuel paths are necessary when constructing the multigraph $\mathscr{G}$. For the rest of this section, the steps of the multigraph construction are explained in detail.  

In general, the construction process is divided into two phases: the network elimination phase and the construction phase. In the elimination phase, some arcs that would not appear in the feasible solution of the problem are eliminated from the original network $G$. The first step is straightforward. The second step was proposed in \citep{schneider2014electric} which involved eliminating arcs that would violate the fuel energy constraint when visiting the two closest CS nodes. 

\begin{enumerate}
    \item Remove from $A$ each arc $(i, j)$ such that $e^{T}_{ij} > Q^{T}$ 
    \item Remove from $A$ each arc $(i, j)$ such that $i \in I, j \in I$, and $e^{T}_{si} + e^{T}_{ij} + e^{T}_{jt} > Q^{T}, \forall \ s, t \in S \cup \{0\}$;
\end{enumerate}

\begin{algorithm}
  \caption{Multigraph construction phase}\label{alg: multigraph_cons}
  \hspace*{\algorithmicindent} \textbf{Input:} Reduced original network $G$\\
  \hspace*{\algorithmicindent} \textbf{Output:} Multigraph $\mathscr{G}$
  \begin{algorithmic}[1]
      \State Construct network $G_{S} = (S, A_{S})$ where $S$ is the set of all CS nodes and $A_{S}$ is the arc such that $(k,l):k,l \in S, e^{T}_{kl} \leq Q^{T}$
      \State Run the Floyd-Warshall algorithm to get all-pair shortest path between any two CS nodes in $G_{S}$
      \State Denote $r^{S}_{kl}$ as the shortest path between two CS nodes $k,l$
      \For{each node pair $(i,j): i,j \in C_{0}, i \neq j$}
          \State Denote the feasible path set from $i$ to $j$ as $\mathscr{P}_{ij}$ and derive it by following steps:
              \begin{itemize}[leftmargin=1cm, font=\small]
                \item \small Check if node $i$ is linked to $j$ directly. If so, add the direct arc to $\mathscr{P}_{ij}$
                \item Check if there exists feasible path from $i$ to $j$ when a CS node $k \in S$ is added between $i$ and $j$ such that $e^{T}_{ik} \leq Q, e^{T}_{kj} \leq Q$. If any, add the feasible path to $\mathscr{P}_{ij}$
                \item Check if there exist refuel path from $i$ to $j$ that traversing shortest path $r^{S}_{kl}$ such that $e^{T}_{ik} \leq Q^{T}, e^{T}_{lj} \leq Q^{T}$. If any, add the refuel path to $\mathscr{P}_{ij}$
             \end{itemize}
          \State Denote $\mathscr{P}^{'}_{ij}$ as the set containing all the non-dominated paths and initialize $\mathscr{P}^{'}_{ij} = \emptyset$
          \State For each feasible path $p \in \mathscr{P}_{ij}$, if no path in $\mathscr{P}^{'}_{ij}$ can dominate $p$, add it to $\mathscr{P}^{'}_{ij}$. If $p$ dominates some path $p^{'} \in \mathscr{P}^{'}_{ij}$, delete $p^{'}$
          \State Add all the remaining paths in $\mathscr{P}^{'}_{ij}$ to the multigraph $\mathscr{G}$ as arc $(i,j,p)$, where $p$ is a unique index so that we can keep track of the exact route in the original network
      \EndFor
      \State \textbf{return} Multigraph $\mathscr{G}$  
  \end{algorithmic}
\end{algorithm}

After the elimination phase,  Algorithm \ref{alg: multigraph_cons} describes the the construction phase. In line 5, we derive all the feasible paths from $i$ to $j$ by checking three situations: the vehicle travels from $i$ to $j$ directly, the vehicle travels from $i$ to $j$ while traversing only one CS node, and the vehicle travels from $i$ to $j$ while traversing two or more CS nodes using $r^{S}_{kl}$. In line 6, we use dominance rule 1 to check if these feasible paths are non-dominated. Besides, in line 8, we need to assign a unique index $p$ to each non-dominated path to derive the route in the original network. Also, note that when adding arcs to $\mathscr{G}$, we need also store the arc's energy requirement, time cost, and remaining energy (if any). \par

It is worth noting that an upper bound on the number of arcs of $\mathscr{G}$ is $(n + 1)^{2}s^{2}k$ where $k$ is the maximum number of nodes of the shortest path in $G_{S}$ between any two stations, and $s$ is the number of CS nodes. However, the number of arcs reduces significantly in practice with dominance rule 1 and Algorithm 1. Based on the numerical analysis, we found that the number of arcs in $\mathscr{G}$ is about 3-8 times that in the original network $G$ (only considering the arcs that traverse two customers nodes). \citep{schneider2014electric} provided a comprehensive analysis of the multigraph $\mathscr{G}$ where the multigraph is compared with the extended graph, which assumes that the vehicle can make at most one refueling stop when traveling between two customers. \par

\subsection{MILP formulation}
With the construction of $\mathscr{G}$, now we are ready to present the arc-based MILP formulation of EVTSPD-FF. Below is a summary of parameters or sets used in the model:

\noindent\begin{longtable*}
{@{}l @{\quad:\quad} p{15cm}@{}}
$e^{T}_{ijp}$  & minimum energy requirement of arc $(i,j,p)$\\
$e^{r}_{ijp}$  & remaining energy of arc $(i,j,p): p > 0$\\
$c^{T}_{ijp}$  & truck travel time cost of arc $(i,j,p)$\\
$c^{D}_{ijp}$  & drone travel time cost of arc $(i,j,p)$\\
$d_{ijk}$ & total drone travel time cost of sortie $(i-j-k)$\\
$D$ & set that contains every feasible sortie $(i-j-k)$\\
\end{longtable*}

\newpage
Below is a summary of the decision variables used in the model:
\noindent\begin{longtable*}
{@{}l @{\quad:\quad} p{12cm}@{}}
$x^{T}_{ijp} \in \{0,1\}$  & Equals one if the truck traverses arc $(i,j,p)$ in the multigraph (no matter if the drone is on-board or not)\\
$x^{D}_{ijp} \in \{0,1\}$  & Equals one if the drone traverses arc $(i,j,p)$ in the multigraph (no matter if it is on-board or not)\\
$y_{ijk} \in \{0,1\}$  & Equals one if the drone is launched from node $i$, travels to node $j$ and returns to the EV at node $k$ and zero otherwise \\
$y^{T}_{i}\in \{0,1\}$ & Equals one if customer node $i$ is a truck customer\\
$y^{D}_{i}\in \{0,1\}$ & Equals one if customer node $i$ is a drone customer\\
$y^{C}_{i}\in \{0,1\}$ & Equals one if customer node $i$ is a combined customer\\
$t_{i}$ & Arrival time of the truck or the drone or both at node $i \in C_{0}$\\
$b_{i}^{a1} \geq 0$ & Theoretical remaining driving energy charge of the EV upon arrival at node $i$, which is measured in time units. \\
$b_{i}^{a2} \geq 0$ & Actual remaining driving energy charge of the EV upon arrival at node $i$, which is measured in time units \\
$b_{i}^{d} \geq 0$ & Remaining battery charge of the EV upon departure from node $i$, which is measured in time units  \\
\end{longtable*}
Here we introduce two driving energy charges upon arrival at node $i$: $b_{i}^{a1}$ and $b_{i}^{a2}$. The first one is a theoretical value which aims to check if the remaining battery charge at previous node $j$ is sufficient to traverse the arc $(j,i,p)$. If $b_{i}^{a1} < 0$, this indicates that the truck cannot traverse the arc $(j,i,p)$ from node $j$.  $b_{i}^{a2}$ is the actual remaining charge upon arrival at node $i$. If arc $(j,i,p): p > 0$ is traversed, then $b_{i}^{a2} = e^{r}_{jip}$. If arc $(j,i,0)$ is traversed, then $b_{i}^{a2} = b_{i}^{a1}$. 

\newpage
\noindent \textbf{\underline{Arc-based model:}} \par 

\vspace{0.2cm} 
\noindent \textbf{Objective} :
\begin{align}
\min \quad
& t_{0^{'}} \label{eqn:arcmodel_obj}
\end{align}

\noindent \textbf{Subject to:}
\begin{align}
\sum_{\substack{j: \\(i,j,p) \in \mathscr{G}}}\sum_{\substack{p: \\ (i,j,p) \in \mathscr{G}}} x^{T}_{ijp} & = \sum_{\substack{j:\\(j,i,p) \in \mathscr{G}}}\sum_{\substack{p: \\(j,i,p) \in \mathscr{G}}} x^{T}_{jip} && \forall \ i \in C \label{eqn:arcmodel_con1} \\
\sum_{\substack{j: \\(i,j,p) \in \mathscr{G}}}\sum_{\substack{p: \\ (i,j,p) \in \mathscr{G}}} x^{T}_{ijp} & = y^{T}_{i} + y^{C}_{i} && \forall \ i \in C \label{eqn:arcmodel_con2} \\
\sum_{\substack{j: \\(0,j,p) \in \mathscr{G}}}\sum_{\substack{p: \\ (0,j,p) \in \mathscr{G}}} x^{T}_{0jp} & = \sum_{\substack{i: \\(i,0^{'},p) \in \mathscr{G}}}\sum_{\substack{p: \\
(i,0^{'},p) \in \mathscr{G}}} x^{T}_{i0^{'}p} = 1  \label{eqn:arcmodel_con3} \\
\sum_{\substack{j: \\(i,j,0) \in \mathscr{G}}}\sum_{\substack{p: \\ (i,j,p) \in \mathscr{G}}}x^{D}_{ijp} & = \sum_{\substack{j: \\(j,i,0) \in \mathscr{G}}} \sum_{\substack{p: \\ (j,i,p) \in \mathscr{G}}} x^{D}_{jip} && \forall \ i \in C  \label{eqn:arcmodel_con4}\\
\sum_{\substack{j: \\(i,j,0) \in \mathscr{G}}}\sum_{\substack{p: \\ (i,j,p) \in \mathscr{G}}}x^{D}_{ijp} & = y^{D}_{i} + y^{C}_{i} && \forall \ i \in C \label{eqn:arcmodel_con5} \\
\sum_{\substack{j \\(0,j,p) \in \mathscr{G}}} \sum_{\substack{p: \\ (0,j,p) \in \mathscr{G}}}x^{D}_{0jp} & = \sum_{\substack{i: \\(i,0^{'},p) \in \mathscr{G}}} \sum_{\substack{p: \\ (i,0^{'},p) \in \mathscr{G}}} x^{D}_{i0^{'}p} = 1 \label{eqn:arcmodel_con6}\\
y^{T}_{i}+ y^{D}_{i}+ y^{C}_{i} &=1 && \forall i \in C  \label{eqn:arcmodel_con7}\\
t_{i}+c^{T}_{ijp} &\leq t_{j} + M(1-x^{T}_{ijp}) && \forall \ (i,j,p) \in \mathscr{E}   \label{eqn:arcmodel_con8}\\
t_{i}+c^{D}_{ijp} &\leq t_{j} + M(1-x^{D}_{ijp}) && \forall \ (i,j,p) \in \mathscr{E}   \label{eqn:arcmodel_con9}\\
x^{D}_{ijp} &\leq y^{C}_{i}+y^{C}_{j} && \forall \ (i,j,p) \in \mathscr{E}  \label{eqn:arcmodel_con10}\\
\sum_{\substack{p: \\ (0,i,p) \in \mathscr{G}}}x^{D}_{0ip}+\sum_{\substack{p: \\ (i,0^{'},p) \in \mathscr{G}}}x^{D}_{i0^{'}p} &\leq 1 && \forall \ i \in C \label{eqn:arcmodel_con11}\\
b^{a1}_{j} &\leq b^{d}_{i} - e^{T}_{ijp}+M(1-x^{T}_{ijp}) && \forall \ (i,j,p) \in \mathscr{E} \label{eqn:arcmodel_con12} \\
b^{a2}_{j} &\leq e^{r}_{ijp}+M(1-x^{T}_{ijp}) && \forall \ (i,j,p), p \neq 0 \in \mathscr{E} \label{eqn:arcmodel_con13} \\
b^{a2}_{j} &\leq b^{a1}_{j}+M(1-x^{T}_{ij0}) && \forall \ (i,j,0) \in \mathscr{E} \label{eqn:arcmodel_con14} \\
b^{a1}_{0} &= b^{a2}_{0} = Q^{T} \label{eqn:arcmodel_con15}\\
b^{d}_{i} &\leq b^{a2}_{i} - \sum_{\substack{j: \\(i,j,k) \in D}}\sum_{\substack{k: \\(i,j,k) \in D}} y_{ijk}d_{ijk} && \forall \ i \in C_{0} \backslash\{0^{'}\}  \label{eqn:arcmodel_con16}\\
\sum_{\substack{j: \\(i,j,k) \in D}}\sum_{\substack{k: \\(i,j,k) \in D}} y_{ijk} &\leq y^{C}_{i} && \forall \ i \in C_{0}\backslash\{0^{'}\} \label{eqn:arcmodel_con17} \\
\sum_{\substack{i: \\(i,j,k) \in D}}\sum_{\substack{k: \\(i,j,k) \in D}} y_{ijk} &= y^{D}_{j} && \forall \ j \in C \label{eqn:arcmodel_con18} \\
\sum_{\substack{i: \\(i,j,k) \in D}}\sum_{\substack{j: \\(i,j,k) \in D}} y_{ijk} &\leq y^{C}_{k} && \forall \ k \in C_{0}\backslash\{0\} \label{eqn:arcmodel_con19} \\
x^{D}_{ij0} + x^{D}_{jk0} &\geq 2y_{ijk} && \forall \ (i,j,k) \in D \label{eqn:arcmodel_con20}
\end{align}

In the above formulation, objective function (\ref{eqn:arcmodel_obj}) aims to minimize the final completion time. Constraints (\ref{eqn:arcmodel_con1}) are truck route flow balance constraints at node $i$. Constraints (\ref{eqn:arcmodel_con2}) state that if node $i$ is visited by the truck then it is either a truck node or a combined node. Constraints (\ref{eqn:arcmodel_con3}) indicate that the truck must depart from and return to the depot. Constraints (\ref{eqn:arcmodel_con4}) are the drone route flow balance constraints at node $i$. Constraints (\ref{eqn:arcmodel_con5}) state that if a drone traverses arc $(i,j,0)$, then node $i$ is either a drone node or combined node. Constraints (\ref{eqn:arcmodel_con6}) indicate that the drone must departs from and returns to the depot. Constraints (\ref{eqn:arcmodel_con7}) state that node $i$ needs to be either truck node, drone node or combined node.  Constraints (\ref{eqn:arcmodel_con8}) and (\ref{eqn:arcmodel_con9}) are the time coordination constraints for the truck and drone, respectively. Constraints (\ref{eqn:arcmodel_con10}) ensure that the drone travels from $i$ to $j$ if and only if the truck visits at least one of the two customers $i$ and $j$. Constraints (\ref{eqn:arcmodel_con11}) prohibit sortie of type $(0,i,0^{'})$. Constraints (\ref{eqn:arcmodel_con12}) is the driving energy constraints of the truck. Constraints (\ref{eqn:arcmodel_con13}) and (\ref{eqn:arcmodel_con14}) calculate the actual arrival battery level if different arcs are traversed. Constraints (\ref{eqn:arcmodel_con15}) indicate that the truck departs from the depot with fully-charged battery. Constraints (\ref{eqn:arcmodel_con16}) are the shared energy constraints. Constraints (\ref{eqn:arcmodel_con17}) -  (\ref{eqn:arcmodel_con20}) are the coordination constraints between $y_{ijk}$ and other decision variables.

\subsection{Handling additional side constraints}
In this section, we aim to relax some of the operational assumptions of EVTSPD-FF and explore how the arc-based model can handle these additional side constraints that are commonly seen in practice and literature. In this section, we only discuss some of the most common ones. For further discussion on this topic, the reader can refer to \citep{Cavani2021}.

\textbf{1. Fixed charging time or service time} \par
Assume that each node $i \in C \cup S$,  is associated with a fixed time cost $f_{i}$. If $i$ is a customer node, $f_{i}$ represents the fixed service time. If $i$ is a charging station node, $f_{i}$ represents the fixed charging time. \par

To account for this fixed node cost in the first model, we can simply add this fixed cost of node $i$ to the arc cost of the node's incoming or outgoing arcs. For example, define the modified travel times $c^{mT}_{ij} = c^{T}_{ij} + f_{i}, \forall \ i \in C \cup S$. We can handle these side constraints by replacing the travel times $c^{T}_{ij}$ with the modified travel times $c^{mT}_{ij}$. 

\textbf{2. Loops}\par
In this case, the launch node and the retrieve node of a drone sortie could be the same node. It indicates that the truck is allowed to wait at the launch node until the drone returns. With this assumption the drone sortie of the form $(i-j-i)$ is allowed which is also called a \textit{loop}.\par
Define a new binary decision variable $z_{ij}$, which equals to 1 if the drone performs loop $( i-j-i)$ and 0 otherwise. To account for the impact of the loop, the objective function is changed to:

\begin{equation}
    t^{*} = min \ t_{0'} + \sum_{j}(c^{D}_{0j}+c^{D}_{j0^{'}})z_{0j} +  \sum_{(i,j):i \neq j}(c^{D}_{ij}+c^{D}_{ji})z_{ij} \label{eqn:loops_obj}
\end{equation}
Constraints (\ref{eqn:arcmodel_con6}) are replaced by:
\begin{equation}
    y^{T}_{i}+ y^{D}_{i}+ y^{C}_{i}+\sum_{j}z_{ji} =1,  \ \forall \ i \in C_{0} \label{eqn:loops_newcon1}
\end{equation}
\begin{equation}
    z_{ij} \leq y^{C}_{i}, \ \forall \ i,j \in C_{0}: i \neq j \label{eqn:loops_newcon2}
\end{equation}
Constraints (\ref{eqn:arcmodel_con15}) are replaced by:
\begin{equation}
    b^{d2}_{i} \leq b^{a2}_{i} - \sum_{j}\sum_{k} y_{ijk}d_{ijk} - \sum_{j:j \neq i}(e^{D}_{ij}+e^{D}_{ji})z_{ij},  \ \forall \ i \in C_{0}  \label{eqn:loops_newcon3}
\end{equation}

The new objective (\ref{eqn:loops_obj}) considers the extra waiting time incurred by the loops. Constraints (\ref{eqn:loops_newcon1}) indicate that a node $i$ can be a truck node, combined node, or a drone node that is visited by a drone sortie or a loop. Constraints (\ref{eqn:loops_newcon2}) states that a node can perform loop only when it is a combined node. Constraints (\ref{eqn:loops_newcon3}) is the updated shared energy constraints that consider the existence of loops.

\textbf{3. Incompatible Customers}\par
A incompatible customer is a customer that can only be served by the truck and not by the drone. Denote $C^{D} \in C$ be the set of customers that can be served by the drone. We can add the following constraints into the first model to account for incompatible customers:

\begin{equation}
    y^{D}_{i} =0,  \ \forall \ i \in C \backslash C^{D} \label{eqn:incompcustomer_newcon1}
\end{equation}
\begin{equation}
    z_{ji} = 0, \ \forall \ i \in C \backslash C^{D},j \in C_{0}: i \neq j \label{eqn:incompcustomer_newcon2}
\end{equation}

\textbf{4. Launch and Retrieve time}\par

Assume that it takes the truck $S_{L}$ time to launch the drone, and during this process, the truck and the drone do not move. This assumption adds extra time cost to each launch node. We further assume that this launch time is not incurred if the launch node is the depot. To model launch time, we define a new binary variable $l_{i}, i \in C_{0}$, which equals one if node $i$ is a drone node and not directly visited by the drone from the depot. Then the objective (\ref{eqn:arcmodel_obj}) is replaced by:
\begin{equation}
   t^{*} = min \ t_{0'} + \sum_{j}(c^{D}_{0j}+c^{D}_{j0^{'}})z_{0j} +  \sum_{(i,j):i \neq j}(c^{D}_{ij}+c^{D}_{ji}+S_{L})z_{ij} + \sum_{i \in C_{0}}S_{L}l_{i} \label{eqn:launch_obj}
\end{equation}
Besides, we need to add constraints:
\begin{equation}
   l_{i} \geq y^{D}_{i} - x^{D}_{0i} \label{eqn:launch_newcon1}
\end{equation}
The new objective (\ref{eqn:launch_obj}) considers the extra launch time of the loops and at each launch node. The constraints (\ref{eqn:launch_newcon1}) states that $l_{i}$ equals one only when node $i$ is a drone node and not directly visited by the drone from the depot.\par

Similarly, assume that it takes the truck $S_{R}$ time to retrieve the drone, and during this process, the truck and the drone do not move. We can model this retrieve time by replacing objective (\ref{eqn:arcmodel_obj}) with the new objective:
\begin{equation}
   t^{*} = min \ t_{0'} + \sum_{j}(c^{D}_{0j}+c^{D}_{j0^{'}}+S_{R})z_{0j} +  \sum_{(i,j):i \neq j}(c^{D}_{ij}+c^{D}_{ji}+S_{R})z_{ij} + \sum_{i \in C_{0}}S_{R}y^{D}_{i} \label{eqn:retrieve_obj}
\end{equation}

\section{Solution method}
Solving the EVTSPD-FF of practical size with a commercial solver using the MILP formulations mentioned in the last section is prohibitive, considering the enormous number of constraints and decision variables involved. This section first introduces an exact branch-and-price solution algorithm. This solution method is similar to the work presented in \citep{roberti2021exact} but with extra complications to consider the fuel energy constraint and the multigraph. The branch-and-price algorithm is based on the set partitioning formulation of the problem. However, in the case of EVTSPD-FF, solving the pricing problem exactly, namely, obtaining a coordinated route with the least reduced cost that visits each customer exactly once while satisfying all the driving/flight range constraints, would be as difficult as solving the original problem. Thus, a relaxation of the EVTSPD-FF route based on the well-known ng-route relaxation \citep{baldacci2011new} is implemented. Using this relaxation, we might derive a "worse" lower bound than solving the pricing problem exactly. However, great computational benefits could be achieved, significantly increasing the algorithm's efficiency. This section presents the set partitioning formulation of EVTSPD-FF, followed by a detailed explanation of the ng-route relaxation adopted when solving the pricing problem using the dynamic programming technique. Then, the branch-and-bound strategy is introduced. \par
Although the BP algorithm is much more efficient than the commercial solver, it can only solve instances containing up to 10 customers within one hour, significantly less than its TSPD counterpart. Thus, at the end of this section, we also introduce a variable neighborhood search (VNS) heuristic, inspired by the work proposed in \citep{campuzano2021multi}, that aims to solve EVTSPD-FF instance of practical size. \par

\subsection{Set partitioning formulation}
An alternative formulation that has been widely used to model vehicle routing problems and its variants is that based on Set Partitioning (SP) or Set Covering (SC), which was initially proposed by \citep{balinski1964integer}. This formulation uses a possibly exponential number of binary variables for each elementary circuit starting from the depot, traversing each customer, and returning to the depot. The use of SP is widely used to tackle any variants of the vehicle routing problem because all the additional side constraints could be defined in a \textit{feasible} route. \par

The set partitioning formulation of EVTSPD-FF is also defined in the multigraph $\mathscr{G}$. Denote the coordinated route as $R$, which includes a truck route $R^{t}$ and drone route $R^{d}$. The truck route $R^{t}$ starts from the origin depot 0, traverses arcs $a_{k} = (i_{k}, i_{k+1}, p_{k}), a_{k} \in \mathscr{G}, k = 0,1,...,r$ and returns to the destination depot. The drone route $R^{d}$ include several drone sorties that has launch node and retrieve node, which are traversed by the truck route.  

The travel time cost of this route is $t(R) = \sum_{k}t(a_{k})+\sum_{k}w_{k} $, namely the sum of the travel times of the arcs it traverses and the extra waiting times at each (retrieve) node. The driving energy $q_{k}(R)$ of this route at each vertex $i_{k}$ is:
\begin{equation}\label{equ:auto_cal}
  q_{k}(R) =
    \begin{cases}
      Q^{T} & \text{if $k=0$}\\
      q_{k-1}(R)- e^{T}_{(i_{k-1}, i_{k}, p_{k-1})} & \text{if $p_{k-1}=0$}\\
      e^{r}_{(i_{k-1}, i_{k}, p_{k-1})} & \text{otherwise}
    \end{cases}       
\end{equation}
In equation (\ref{equ:auto_cal}), when the truck/drone departs origin depot 0, its remaining energy is $Q^{T}$. If the truck traverses an arc $(i_{k-1}, i_{k}, 0)$, then the remaining  energy at node $i_{k}$ is the remaining energy at node $i_{k-1}$ minus the consumed energy along the arc. If the truck traverses an arc $(i_{k-1}, i_{k}, p_{k-1})$ where $p_{k-1} \neq 0$, it means that the truck traverses refuel path $p_{k-1}$ from $i_{k-1}$ to $i_{k}$. Thus the remaining energy at node $i_{k}$ is a fixed value $e^{r}_{(i_{k-1}, i_{k}, p_{k-1})}$ associated with the traversed arc.

Meanwhile, a route $R$ is feasible if the following conditions are satisfied:
\begin{enumerate}
    \item $0 \leq q_{k}(R) \leq Q^{T}, \forall k=0,1,...,r+1$
    \item $q_{k} \geq e^{T}_{(i_{k-1}, i_{k}, p_{k-1})}, \forall k=0,1,...,r$
\end{enumerate}

These conditions indicate that whenever it aims to traverse an arc, the remaining energy of the truck should be no less than the required energy of the arc. 

Let $\mathscr{R}$ be the index set of all the feasible routes of EVTSPD-FF in $\mathscr{G}$. Thus, each feasible route consists of an ordered sequence of operations and combined legs (see their definitions in Section 3). Each such route $r \in \mathscr{R}$ is characterized by coefficients $a_{ir}$ that indicate the number of times customer $i \in C$ is visited in route $r$ and by $d_{r}$ that represents the total time cost of route $r$. Let $\lambda_{r}$ be the binary decision variable that takes the value 1 if and only if route $r$ is in the solution. The EVTSPD-FF could be modeled as the following set partitioning problem:
\hfill \break
\begin{align}
[SP] \hspace{0.5cm}  \min \quad & \sum_{r \in \mathscr{R}} d_{r}\lambda_{r} \label{sp_obj}\\
s.t. & \sum_{r \in \mathscr{R}}\lambda_{r} =1 \label{sp_con1}\\
& \sum_{r \in \mathscr{R}} a_{ir}\lambda_{r} = 1, \forall i \in C \label{sp_con2}\\
& \lambda_{r} \in \{0,1\}, \forall r \in \mathscr{R} \label{sp_con3}
\end{align}
The objective function (\ref{sp_obj}) aims at minimizing the duration of the selected route. Constraint (\ref{sp_con1}) indicates the only one of such route can be selected.  Constraints (\ref{sp_con2}) guarantees each customer be visited exactly once and constraint (\ref{sp_con3}) is the domain constraint for decision variables.\par

Solving the formulation SP via a commercial solver would be prohibitive as there is an exponential number of feasible routes in the set $\mathscr{R}$. The most common approach to solving SP is to resort to the column generation (CG) or branch-and-price algorithm and generate feasible routes dynamically. By replacing the full set $\mathscr{R}$ with the dynamically computed restricted route set $\mathscr{R}^{i}$, where $i$ is the index of the iterations, the linear relaxation of SP formulation is called the master problem, or in short, MP. The process of finding a possible route and adding it to $\mathscr{R}^{i}$ is called a pricing problem. The basic idea of CG or BP is to find any basic feasible solutions in the pricing problem using the dual variable of the master problem and add such solutions to the restricted route set $\mathscr{R}^{i}$ of the master problem until this process can be no longer continued. By replacing the full set $\mathscr{R}$ with the restricted route set $\mathscr{R}^{i}$ and solving MP, we get a lower bound of the EVTSPD-FF as it is a relaxation of the original problem. Then the branch-and-bound process is carried on to obtain the binary solution of SP and hence the optimal solution of EVTSPD-FF. \par

The critical step of this solution algorithm is to generate basic feasible solutions in the pricing problem, i.e., a route that visits each customer exactly once and satisfies all the drone and truck's related energy capacity constraints. Sadly, finding such a route is as complex as solving the original problem considering all the constraints involved. Thus, a relaxation technique is required to easily generate feasible routes at the cost of increasing the number of iterations.

\subsection{ng-route relaxation}
One of the most efficient route relaxations in solving routing problems and its variants is the ng-route relaxation proposed in \citep{baldacci2011new}. An ng-route is not necessary an elementary route that visits each customer exactly once. In an ng-route, a customer could be visited multiple times, while some may not be visited at all. When a partial ng-route propagates to the next customer, one cannot visit if this customer lies in a dynamically computed set called ng-set. In other words, the ng-route can visit a customer $i$ multiple times if, in between these visits, the route visits another customer $j$ that is far away from $i$.\par

Using the ng-route relaxation in the pricing problem of BP indicates that some ng-route routes that are not necessarily elementary are obtained in each iteration, which would be added back into the restricted route set $\mathscr{R}^{i}$ of the master problem. Solving the master problem of SP enables us to get a lower bound of the original problem. Compared to solving the pricing problem exactly without using any relaxation technique, ng-route relaxation would result in a "bad" route which might lead to deteriorated lower bounding in each iteration of BP. However, the drawback of using this relaxation is not significant because most ng-route is typically cost-inefficient and unattractive and thus unlikely to appear in the MP solution, and only valuable ng-route would appear in the MP solution. Meanwhile, it is much easier to find such least-reduced-cost ng-routes than finding the least-reduced-cost elementary EVTSPD-FF route in the pricing problem, as few state variables are needed, and some strong dominance rules could be applied to reduce the state that needs to be checked. In general, adopting ng-route relaxation can lead to a more efficient algorithm. The rest of this section will illustrate how this technique solves EVTSPD-FF. \par

For each customer $i \in C$, define the ng-set of customer $i$ as the set that contains all the neighbors of customer $i$, i.e., all the customers near $i$. This ng-set indicates all the customers who cannot visit between two consecutive visit customer $i$. The size of the ng-set determines the propagating direction of ng-routes. A larger ng-set indicates greater sub-tour allowed and better lower bounding from solving the master problem. Meanwhile, finding a least-cost ng-route with a greater ng-set size is more complicated. On the other hand, a small ng-set would result in an easy-to-solve pricing problem but may lead to a relatively worse lower bound. Based on the research conducted in \citep{roberti2021exact}, setting the size of ng-set as 5 is a good trade-off between the quality of bounding and the difficulty of solving the pricing problem. This research tested three different ng-set sizes and reported their computational results in the numerical analysis section. \par

\subsection{Dynamic programming procedure to solve pricing problem}
One critical step in the branch-and-price algorithm is to obtain the least-reduced-cost route given the dual information of the master problem solution. As mentioned above, the ng-route relaxation technique is used so that the non-elementary route is allowed to enter the restricted route set $\mathscr{R}^{i}$ of the master problem. As a result, in the pricing problem of the BP algorithm, we aim to find a least-reduced-cost ng-route. This goal is accomplished by the dynamic programming (DP) algorithm introduced in this section. This algorithm combines elements of the solution method proposed in \citep{roberti2021exact} and \citep{andelmin2017exact} with necessary modifications made to fit our purpose. The difference between these three DP algorithms will be discussed later in this section.\par

The problem description section demonstrates how an EVTSPD-FF solution route can be constructed to concatenate operations and combined legs. Remember that each such operation and combined leg serve their own customers. We define a \textit{feasible} ng-route as a route whose total number of customer visits across all operations and combined legs is $c$ where $c$ is the total number of customers in the instance while satisfying the drone and truck's flight/driving energy limit constraints. The DP recursion used in this research is based on the idea of propagating the current partial route to the next available customer until the total customer visits equal to $c$, and then the route returns to the destination depot. Such route is constructed in the multigraph $\mathscr{G}$  and can propagate to the next customer by three possible actions: adding a truck arc, a drone leg, or a combined arc. By deploying these three actions, all the feasible ng-routes could be generated, and the least-cost route could be identified and added to $\mathscr{R}^{i}$. In this process, some strong dominance rules could be used to eliminate unpromising and dominated routes and thus facilitate the procedure. \par

Let $u_{0} \in \mathbb{R}$ and $u_{i} \in \mathbb{R}, i \in N$ be the dual variables associated with constraints (\ref{sp_con1}) and (\ref{sp_con2}), respectively of the linear relaxation of SP formulation. Thus, in the pricing problem, the objective function is $d_{r} - u_{o} - \sum_{i \in N}a_{ir}u_{i}$ where $d_{r}$ is the route duration of route $r$ and $a_{ir}$ is a parameter indicating the number of times customer $i$ is visited in route $r$. By solving the pricing problem, we aim to find an ng-route with the least objective value. In DP, we define a function 
$$f(ng, k, i^{C}, i^{T}, \tau, b^{C}, b^{T}, Isbcfixed)$$ as the representation of a partial ng-route. The elements of this function are described as follows:

\begin{itemize}
    \item $ng$ is the ng-set of the route, which indicates the set of customers that cannot be visited in the next propagation. It is a subset of customer set and also a subset of $N_{i^{T}}$, ng-set of customer $i^{T}$.
    \item $k$ represents the number of customer visits performed so far. When $k$ equals $c$, we either terminate the route or force the route to return to the depot.
    \item $i^{C} \in N$ is the last combined node visited by the truck in the partial route.
    \item $i^{T} \in N$ is the last customer node visited by the truck (regardless of whether the drone is on-board or not) in the partial route.
    \item $\tau \geq 0$ represents the truck's travel time since the last time the truck visits $i^{C}$.
    \item $b^{C}$ represents the remaining driving energy of the truck at node $i^{C}$.
    \item $b^{T}$ represents the remaining driving energy of the truck at node $i^{T}$.
    \item $Isbcfixed$ is a binary indicator that represents whether $b^{c}$ is a fixed value when adding a truck arc. If the truck visit some CS nodes between $i^{C}$ and $i^{T}$, then when adding another truck arc, the value $b^{c}$ will not change and $Isbcfixed = True$. If no CS nodes exist between $i^{C}$ and $i^{T}$, then $b^{C} = b^{T}$ and $Isbcfixed = False$. 
\end{itemize}
\par
As we can see, we need to keep track of the remaining driving energy of the truck at both $i^{C}$ and $i^{T}$. If $i^{C} = i^{T}$, it means that the partial route finishes an operation or traverses a combined arc. In this situation, $b^{C}$ equals $b^{T}$ and represents the maximum driving energy without charging. When $i^{C} \neq i^{T}$, it means that the route has an ongoing operation and $b^{C} \neq b^{T}$. If the route visits next customer $j$ by truck while traversing arc $(i^{T}, j, p) \in \mathscr{G}$, we need to check if remaining energy at $i^{T}$ is sufficient to traverse this arc, that is, check if $b^{T} \geq e^{T}_{(i^{T}, j, p)}$. Otherwise, if the route aims to visit customer $j$ by drone, it indicates that a drone leg $i^{C}-j-i^{T}$ would be performed. Because of the shared energy assumption, we need to check if $b^{C} \geq e^{D}_{i^{C}j} + e^{D}_{ji^{T}}$. Meanwhile, the binary parameter $Isbcfixed$ helps us to keep track of the value of $b^{c}$.  \par

The value of $f(ng, k, i^{C}, i^{T}, \tau, b^{C}, b^{T}, Isbcfixed)$ represents the minimal duration of any partial route that starts from the depot and can be characterized by the tuple $(ng, k, i^{C}, i^{T}, \tau, b^{C}, b^{T}, Isbcfixed)$. \par

In the beginning of the DP, the route is initialized as $f(\emptyset, 0, 0, 0, 0, Q, Q, False) = -u_{0}$. In the propagation process, given the current state $f(ng, k, i^{C}, i^{T}, \tau, b^{C}, b^{T}, Isbcfixed)$, there are three different actions can be taken based on the current partial route state:
\begin{enumerate}
\setlength\itemsep{1em}
    \item Add truck arc: the current partial route will visit the next customer $j \in N / (ng \cup \{i^{T}, i^{C}\})$ by truck only. However, this action can be taken if and only if the following conditions are satisfied:
    \begin{itemize}
        \item There exist an arc $(i^{T}, j, p)$ in $\mathscr{G}$ to connect $i^{T}$ and $j$.
        \item $b^{T} \geq c_{1}(i^{T}, j, p)$, where $c_{1}(i^{T}, j, p)$ is the driving energy requirement for traversing arc $(i^{T}, j, p)$ in $\mathscr{G}$.
    \end{itemize}
     If customer $j$ can be reached by adding a truck arc $(i^{T},j,p)$, the state function for the new partial route is $$f((ng \cup \{i^{T}\}) \cap N_{j}, k+1, j, i^{C}, \tau+c(i^{T},j,p), b^{C'}, b^{T'}, newIsbcfixed)$$ where
     
     \begin{itemize}
         \item $(ng \cup \{i^{T}\}) \cap N_{j}$ is newly the updated ng-set.
         \item $k+1$ is the new customer visit, $j$ is the new $i^{T}$, $i^{C}$ remains the same, $\tau+c(i^{T},j,p)$ is the newly updated travel time since the truck last visit to $i^{C}$.
         \item $b^{C'}$ is the new remaining energy at $i^{C}$. It equals to $b^{C}$ if $Isbcfixed$ is True and $b^{C}-c(i^{T},j,p)$ if $Isbcfixed$ is False.
         \item  $b^{T'}$ is the new remaining driving energy at $i^{T}$ and it equals to $b^{T}-c(i^{T},j,p)$ if $p = 0$ and $e^{r}_{(i^{T},j,p)}$ otherwise.
         \item $newIsbcfixed$ is the new indicator of whether $b^{C'}$ can be further changed. It is True if $p=0$ and False if $p \neq 0$.
     \end{itemize}
    After the propagation, the new value for the state is $f(ng, k, i^{C}, i^{T}, \tau, b^{C}, b^{T}, Isbcfixed) - u_{j}$.

    \item Add combined arc: if $i^{T} = i^{C}$ and $\tau =0$, it means that the current partial route just finishes an operation or combined arc. In this situation, the current partial route can visit the next customer $j \in N \textbackslash (ng \cup \{i^{T}, i^{C}\}$ by adding a combined arc,  if and only if the following conditions are satisfied:
    \begin{itemize}
        \item There exist an arc in $\mathscr{G}$ to connect $i^{T}$ and $j$.
        \item $b^{T} \geq c_{1}(i^{T}, j, p)$, where $c_{1}(i^{T}, j, p)$ is the driving energy requirement for traversing arc $(i^{T}, j, p)$ in $\mathscr{G}$.
    \end{itemize}
    Similarly, if customer $j$ can be reached by adding a combined arc $(i^{T},j,p)$, the state function for the new partial route is $$f((ng \cup \{i^{T}\}) \cap N_{j}, k+1, j, j, 0, b^{C'}, b^{T'}, False)$$ where 
    \begin{enumerate}
         \item $(ng \cup \{i^{T}\}) \cap N_{j}$ is newly the updated ng-set.
         \item $k+1$ is the new customer visit, $j$ is the new $i^{T}$ and $i^{C}$, the newly updated travel time since the truck last visit to $i^{C}$ is still zero.
         \item $b^{C'}$ is the new remaining energy at $i^{C}$. It equals to $b^{C}$ if $b^{C}-c(i^{T},j,p)$ if $p = 0$ and $e^{r}_{(i^{T},j,p)}$ otherwise.
         \item  $b^{T'}$ is the new remaining driving energy at $i^{T}$ and it equals to $b^{T}-c(i^{T},j,p)$ if $p = 0$ and $e^{r}_{(i^{T},j,p)}$ otherwise.
         \item The new indicator of whether $b^{C'}$ is fixed is False. 
     \end{enumerate}
    
     The value for the new state  is $f(ng, k, i^{C}, i^{T}, \tau, b^{C}, b^{T}, False) - u_{j} + c(i^{T},j,p)$ where $c(i^{T},j,p)$ is the time cost of arc $(i^{T},j,p)$ in $\mathscr{G}$.

    \item Add drone leg: if $i^{T} \neq i^{C}$ and consequently $\tau \neq 0$, the current partial route has an ongoing unfinished operation. In this situation, the current partial route can visit the next customer $j \in N \textbackslash (ng \cup \{i^{T}, i^{C}\}$ by adding a drone sortie $(i^{C}, j, i^{T})$ which finishes the current operation. This action can only be taken if and only if the following conditions are satisfied: 
    \begin{itemize}
        \item There exist arcs $(i^{C}, j, 0)$ and $(j, i^{T}, 0)$ in $\mathscr{G}$ such that $e^{D}_{(i^{C}, j, 0)} + e^{D}_{(j, i^{T}, 0)} \leq Q^{d}$, which requires that the sortie respect the drone flight energy capacity constraint. 
        \item $e^{D}_{(i^{C}, j, 0)} + e^{D}_{(j, i^{T}, 0)} \leq b^{C}$ because of the shared energy assumption. This condition ensures that when the drone is launched at node $i^{C}$, the extra remaining energy at $i^{C}$ should be no less than the required energy of the sortie.
    \end{itemize}
    If the drone leg is added successfully, the state function for the new partial route is $$f((ng \cup \{i^{T}, i^{C}\}) \cap N_{j}, k+1, i^{T}, i^{T}, 0, b^{T}, b^{T}, False).$$ The value for this new state is $f(ng, k, i^{C}, i^{T}, 0, b^{C}, b^{T}, Isbcfinal) + max\{\tau, t^{D}_{(i^{C}, j, 0)} + t^{D}_{(j, i^{T}, 0)}\} - u_{j}$, where $max\{\tau, t^{D}_{(i^{C}, j, 0)} + t^{D}_{(j, i^{T}, 0)}\}$ is the time cost for the new operation. 
\end{enumerate}
\vspace{1em}

\begin{figure}
    \centering
    \includegraphics[width = .8\textwidth]{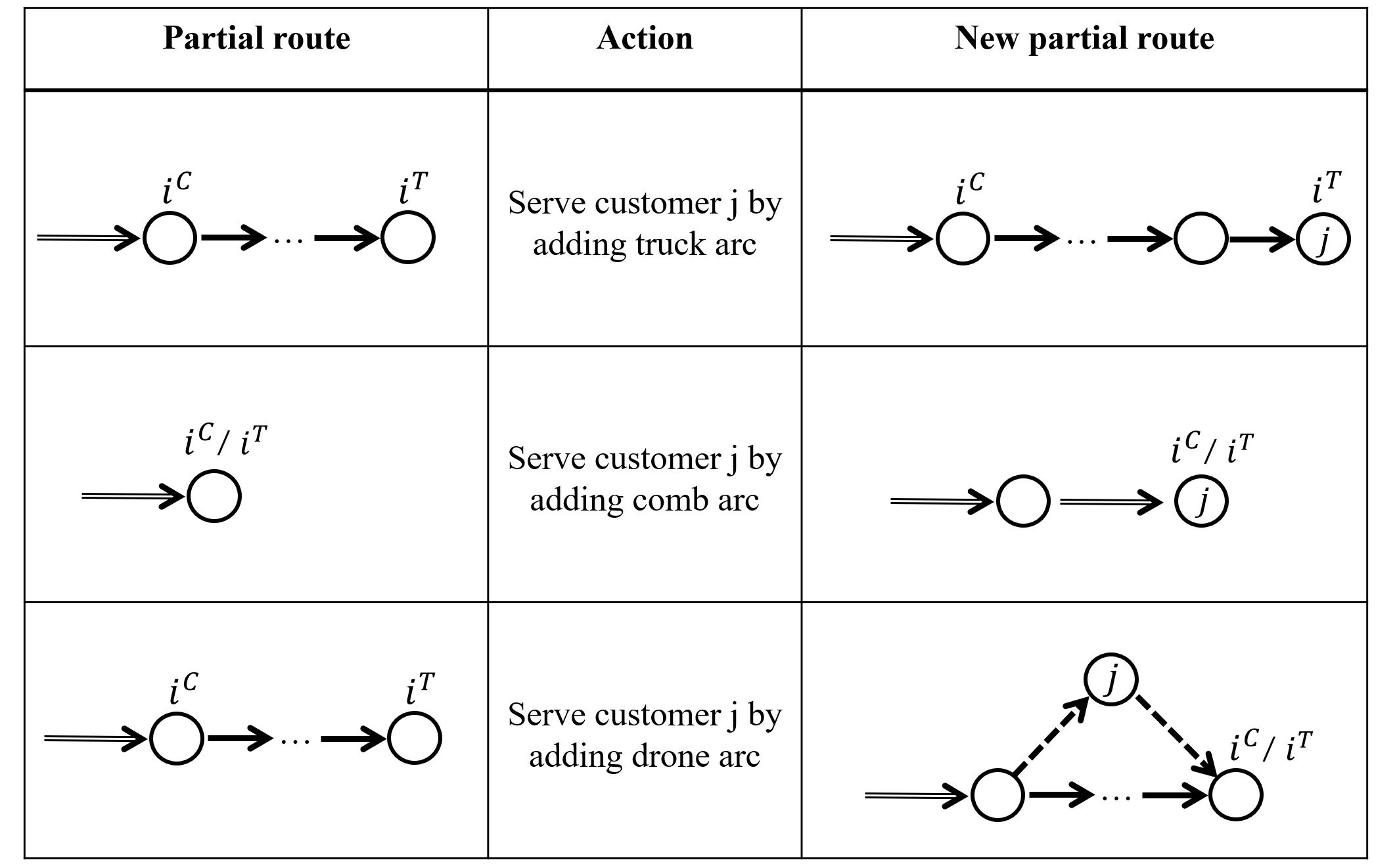}
    \caption{A representation of the three actions in DP}
    \label{fig:DP_actions}
\end{figure}

A simple example is shown in Figure \ref{fig:DP_actions} to illustrate the three actions and their effect on the partial route introduced above. Note that in this figure, a double line arrow represents a combined arc and a single line arrow represents a truck arc. The three actions are deployed at every partial node to progress to the next customer node until the partial route visits exact $n$ customers. If the partial route has to visit $n$ customers, we also need to check if the route has returned to the depot, and an extra arc needs to be added if necessary. In this way, by exploring all the possible partial routes, we obtain an ng-route with the least reduced cost given the dual information of the master problem.\par

However, the number of possible states still grows exponentially with the instance size, especially considering that multiple arcs exist between two nodes in the multigraph, and three possible actions exist. Thus, the following dominance rules are used to limit the number of states in the algorithm. During the algorithm process, these two dominance rules are constantly called to check if the new state dominates or is dominated by others. 

\begin{theorem}
Let $f(ng_{1}, k, i^{C}, i^{T}, \tau, b^{C}_{1}, b^{T}_{1}, Isbcfinal) = f^{1}$ and $f(ng_{2}, k, i^{C}, i^{T}, \tau, b^{C}_{2}, b^{T}_{2}, Isbcfinal) = f^{2}$. If the following conditions are satisfied:

\begin{subequations}
    \begin{align}
        f^{1} \leq f^{2} \\
        i^{C} = i^{T} \\
        ng_{1} \subseteq ng_{2} \\
        b^{C}_{1} \geq b^{C}_{2} \\
        b^{T}_{1} \geq b^{T}_{2} 
    \end{align}
\end{subequations}
Then $f(ng_{1}, k, i^{C}, i^{T}, \tau, b^{C}_{1}, b^{T}_{1}, Isbcfinal)$ dominates $f(ng_{2}, k, i^{C}, i^{T}, \tau, b^{C}_{2}, b^{T}_{2}, Isbcfinal)$
\end{theorem}

\begin{proof}
If all the five above conditions are satisfied,  all the potential states that could be achieved by propagating $f(ng_{2}, k, i^{C}, i^{T}, \tau, b^{C}_{2}, b^{T}_{2}, Isbcfinal)$ can also be achieved by propagating $f(ng_{1}, k, i^{C}, i^{T}, \tau, b^{C}_{1}, b^{T}_{1}, Isbcfinal)$ with a lower cost. Thus the dominance rule stands.
\end{proof}

\begin{theorem}
Let $f(ng_{1}, k, i^{C}, i^{T}, \tau, b^{C}_{1}, b^{T}_{1}, Isbcfinal) = f^{1}$ and $f(ng_{2}, k, i^{C}, i^{T}, \tau, b^{C}_{2}, b^{T}_{2}, Isbcfinal) = f^{2}$. If the following conditions are satisfied:
\begin{subequations}\label{equ:domin3}
    \begin{align}
        f^{1} \leq f^{2}  \label{domin3_1} \\ 
        i^{C} \neq i^{T} \label{domin3_2} \\
        ng_{1} \subseteq ng_{2} \label{domin3_31} \\
        b^{C}_{1} \geq b^{C}_{2} \label{domin3_4} \\
        b^{T}_{1} \geq b^{T}_{2} \label{domin3_5} \\
        f^{1}+\tau_{1} \leq f^{2}+\tau_{2} \label{domin3_6}
    \end{align}
\end{subequations}
Then $f(ng_{1}, k, i^{C}, i^{T}, \tau, b^{C}_{1}, b^{T}_{1}, Isbcfinal)$ dominates $f(ng_{2}, k, i^{C}, i^{T}, \tau, b^{C}_{2}, b^{T}_{2}, Isbcfinal)$
\end{theorem}

\begin{proof}
Define a reverse ng-route that starts with the destination depot while back propagating to origin depot. Thus we could also define a function $f(ng_{b}, n-k, i^{C}_{b}, i^{T}, \tau_{b}, b^{C}_{b}, b^{T}, Isbcfinal) = f^{b}$ as the representation of this reverse partial ng-route. In this way, a full ng-route could be achieved by combining $f(ng, k, i^{C}, i^{T}, \tau, b^{C}, b^{T}, Isbcfinal)$ with $f(ng_{b}, n-k, i^{C}_{b}, i^{T}, \tau_{b}, b^{C}_{b}, b^{T}, Isbcfinal)$.\par

If all the above conditions are satisfied, all the potential propagation that could be achieved by propagating $f(ng_{2}, k, i^{C}, i^{T}, \tau, b^{C}_{2}, b^{T}_{2}, Isbcfinal)$ can also be achieved by propagating $f(ng_{1}, k, i^{C}, i^{T}, \tau, b^{C}_{1}, b^{T}_{1}, Isbcfinal)$. \par

By combining $f^{1}$ with backward partial route, we can derive a full ng-route. Assuming a route performs drone leg $i^{C}-j-i^{C}_{b}$ where $j$ is the drone node, denote $t_{cb} = t^{D}_{i^{C}j} + t^{D}_{ji^{C}_{b}} + f_{j}$ as the needed time for drone to finish the this drone leg. Then the reduced cost of the full route is:
$$\hat{d_{1}} = f_{1}+f^{b}+u_{i^{T}}-u_{j}+max\{\tau_{1}+\tau_{b}, t_{cb}\} $$
Similarly, the reduced cost of the full route by combining $f^{2}$ with $f^{b}$ is 
$$\hat{d_{2}} = f_{2}+f^{b}+u_{i^{T}}-u_{j}+max\{\tau_{2}+\tau_{b}, t_{cb}\} $$

If all the conditions from (\ref{domin3_1}) to (\ref{domin3_6}) are satisfied, we need to prove that $\hat{d_{1}} \leq \hat{d_{2}}$. This can be illustrated in three cases:

\begin{itemize}
    \item If $t_{cb} \geq max\{\tau_{1}+\tau_{b}, \tau_{2}+\tau_{b}\}$, then $\hat{d_{1}} = f_{1}+f^{b}+u_{i^{T}}-u_{j}+ t_{cb}$ and $\hat{d_{2}} = f_{2}+f^{b}+u_{i^{T}}-u_{j}+ t_{cb}$. Thus $\hat{d_{1}} \leq \hat{d_{2}}$ because of condition (\ref{domin3_1}).
    \item If $t_{cb} \leq min\{\tau_{1}+\tau_{b}, \tau_{2}+\tau_{b}\}$, then $\hat{d_{1}} = f_{1}+f^{b}+u_{i^{T}}-u_{j}+ \tau_{1}+\tau_{b}$ and $\hat{d_{2}} = f_{2}+f^{b}+u_{i^{T}}-u_{j}+\tau_{2}+\tau_{b}$. Thus $\hat{d_{1}} \leq \hat{d_{2}}$ because of condition (\ref{domin3_6}).
    \item If $\tau_{1}+\tau_{b} \leq t_{cb} \leq \tau_{2}+\tau_{b}$, then  $\hat{d_{1}} = f_{1}+f^{b}+u_{i^{T}}-u_{j}+ t_{cb}$ and $\hat{d_{2}} = f_{2}+f^{b}+u_{i^{T}}-u_{j}+ \tau_{2}+\tau_{b}$. Thus $\hat{d_{1}} \leq \hat{d_{2}}$ because of condition (\ref{domin3_1}).
\end{itemize}

Thus, $\hat{d_{1}} \leq \hat{d_{2}}$ in all the scenarios and the dominance rule stands.

\end{proof}

\subsection{Considering additional side constraints in dynamic programming}

\subsubsection{Launch and retrieve times}
The launch and retrieve time assumption can be incorporated in the DP model by changing the state value calculation to account for the launch and retrieve times. When adding a truck arc, the new state value is $f(ng, k, i^{C}, i^{T}, \tau, b^{C}, b^{T}, Isbcfixed) - u_{j} + S_{L}$. When adding a drone leg, the new state value is $f(ng, k, i^{C}, i^{T}, \tau, b^{C}, b^{T}, False) - u_{j} + c(i^{T},j,p) + S_{R}$.

\subsubsection{Maximum number of customers per truck leg}
In dynamic programming, we can add a truck arc when $i^{T} != i^{C}$ and there is enough energy to traverse the next truck arc. When the maximum number of customers per truck leg assumption is enforced, we need to check an additional condition, that is, the current number of customers in the current truck leg, to ensure a feasible route. We need an extra element to record this value when constructing states in DP to do so. Define a new state variable $n^{T}$ indicating the current number of customers visited in the current truck leg. In this way, the new state is
$$f(ng, k, i^{C}, i^{T}, \tau, b^{C}, b^{T}, Isbcfixed, n^{T}).$$\par

Note that it might seem non-trivial to add a state variable in DP as this would lead to expanded state space. However, in practice, when the upper bound $\Bar{n}$ is small (e.g. $\Bar{n} \leq 3$), the influence of the new state variable is minimal. Besides, as $\Bar{n}$ places an upper bound of the number of consecutive "add truck arc" actions, the propagation process and hence the overall computation process is shortened, as suggested in the numerical analysis section.

\subsubsection{Weight-dependent drone flying range}
The weight-dependent drone flying range assumption can be enforced in DP by changing the original constraint $e^{D}_{(i^{C}, j, 0)} + e^{D}_{(j, i^{T}, 0)} \leq Q^{d}$ when adding drone leg to the new constraint $f^{D}_{(i^{C}, j, 0)} + f^{D}_{(j, i^{T}, 0)} \leq Q^{d}$ where  $f^{D}_{(i, j, p)}$ is the energy consumption function for drones when traversing arc $(i,j,p)$.

\subsubsection{Loops}
When loop is allowed in the operation process, in DP, an extra action of adding loop is needed. For current state $f(ng, k, i^{C}, i^{T}, \tau, b^{C}, b^{T}, Isbcfixed)$, the \textit{add loop} action is only available when $i^{C} = i^{T}, \tau=0, Isbcfixed=False, b^{C}=b^{T}$, which indicates that both the truck and the drone are at node $i^{C}$ at current state, and $b^{C} \geq e^{D}_{i^{T}j}+e^{D}_{ji^{T}}$ which indicates that the remaining energy is sufficient for drone to perform sortie $\textlangle i^{T}, j, i^{T} \textrangle$. After the \textit{add loop} action, the resulting new state is $f((ng \cup \{i^{T}, i^{C}\}) \cap N_{j}, k+1, i^{T}, i^{T}, 0, b^{T}, b^{T}, False)$ and the new state value is $f(ng, k, i^{C}, i^{T}, 0, b^{C}, b^{T}, Isbcfinal) + t^{D}_{(i^{C}, j, 0)} + t^{D}_{(j, i^{T}, 0)} - u_{j}$.\par

\color{black}

\subsection{Branching strategy}
With the ng-route relaxation technique, the columns (routes) entered into the master problem of BP might not be elementary routes. As a result, the final solution of the master problem might be fractional. If this situation happens, a branch-and-bound procedure is carried out to obtain a feasible integer solution, and thus a branching strategy guiding this process is needed. For a node in the B\&B tree, given an optimal solution $\lambda^{*}$ of the master problem of SP formulation and its corresponding column set (ng-route set), we follow the routine introduced in \citep{roberti2021exact} and perform a binary search based on three different types of decisions: (1)$y_{i}^{*}, i \in N$, which represents the times each customer $i$ is served by the drone alone in the optimal solution, (2) $x_{ijp}^{*T}, (i,j,p) \in \mathscr{G}$, which represents the times each arc $(i,j,p)$ is traversed by the truck only in the optimal solution, and (3) $w_{ij}^{*D}, (i,j) \in G$, which represents the times each arc $(i,j)$ is traversed by the drone only in the optimal solution. Note that the value of $y_{i}^{*}, x_{ijp}^{*T}, w_{ij}^{*D}$ should all lie in the interval $[0,1]$ because of constraints (\ref{sp_con3}). \par

With the three kinds of optimal solution values, we can perform branching based on these values in a hierarchical manner. First of all, we check if $y_{i}^{*}, i \in N$ is binary for each customer $i$. If there exists some customer whose $y_{i}^{*}$ value is not binary, we then find the one whose value is closest to 0.5, that is, $i^{*} = argmin_{i \in C}|y_{i}^{*}-0.5|$ and branch on this customer. If all the values of $y_{i}^{*}$ are binary, we can move to the values of $x_{ijp}^{*T}, (i,j,p) \in \mathscr{G}$ and similarly branch on the truck arc whose value is closest to 0.5, that is, $(i,j,p)^{*} = argmin_{(i,j,p) \in \mathscr{G}}|x_{ijp}^{*T}-0.5|$. If all the values of $x_{ijp}^{*T}$ are either zero or one, we then move to the last layer of branching decisions $w_{ij}^{*D}, (i,j) \in G$ and branch on the drone arc whose value is the closest to 0.5, that is, $(i,j)^{*} = argmin_{(i,j) \in G} |w_{ij}^{*D}-0.5|$. Note that we end up with two child nodes with corresponding branching decision values at zero or one for each of these branching decisions. After performing all three layers of branching procedures, we can obtain all the feasible integral solutions to the original EVTSPD-FF and pick the least-cost one as the optimal solution. To increase the efficiency of this branching process, one can apply a depth-first-search method to obtain an integral solution as soon as possible so that it can be used as an upper bound to prune the remaining nodes in the later process. \par 

As the solution EVTSPD-FF solely consists of truck arcs and drone legs, it seems unnecessary to branch on the value of $y_{i}^{*}, i \in C$. However, as explained in \citep{roberti2021exact}, only branching on $x_{ijp}^{*T}$ and $w_{ij}^{*D}$ might lead to an unbalanced search tree as it is more restricting to set the value of a specific $x_{ijp}^{*T}$ or $w_{ij}^{*D}$ to one than setting the value to zero. Besides, based on the numerical results conducted in this paper, we find that for most small-sized cases, merely branching on $y_{i}^{*}, i \in C$ is sufficient to obtain an integral solution. For these reasons, the first layer of the branching decision is on $y_{i}^{*}, i \in C$.

\subsection{Variable neighbourhood search heuristic}
VNS scheme is a classical meta-heuristic that aims to escape local optima by changing the way of selecting neighbourhood structure during the process. This method is first proposed in \citep{mladenovic1997variable} and later improved in \citep{hansen2001variable}. Over the past few decades, this approach has been successfully applied in solving multiple problems, especially in the field of optimization. In the VRP community, it can be used to solve different variant of the problem such as TSP, VRP with time window and VRP with multi-depot. \citep{campuzano2021multi} compare the performance of VNS with the MILP formulation proposed in \citep{Cavani2021} to solve TSPD. Thus, in this section, we also present a similar VNS scheme to solve EVTSPD-FF instance of larger size. The general framework of VNS is shown in Algorithm \ref{alg: vns}.\par 

\begin{algorithm}
\caption{Variable Neighbourhood Search}\label{alg: vns}
\hspace*{\algorithmicindent} \textbf{Input:} EVTSPD-FF network $G$\\
\hspace*{\algorithmicindent} \textbf{Output:} EVTSPD-FF solution $S^{*}$
\begin{algorithmic}[1]
\State \textbf{Initialization:}
\State S $\gets GetIniSol()$
\State $S_{i}, S^{'}_{i}, S^{*}_{i} \gets MakeFly()$       \Comment{Improve the solution}
\State $S_{i}, S^{'}_{i}, S^{*}_{i} \gets GainFeasibility()$  \Comment{Make the solution feasible by inserting charging station}
\State $f(S^{*}) \gets \infty$
\State $i, p \gets 0$
\State \textbf{Iteration:}
\While{$i< MaxIteration$ or $p < MaxStopping$}
    \State $l \gets 0$
    \While{$f(S_{i}) \geq f(S^{'})$ and $l \leq l_{m}$}
        \State $S_{i} \gets N_{l}(S_{i})$  \Comment{Select neighbourhood structure based on $l$}
        \State $l \gets l+1$
        \If{$f(S_{i}) < f(S_{i}^{'})$}
            \If{$f(S_{i}) < f(S_{i}^{*})$}
                \State $S^{*} \gets S_{i}$
                \State $p \gets 0$   \Comment{If find a global best solution then set $p$ to zero}
            \EndIf
        \State $l \gets 0$
        \Else 
        \State $S_{i} \gets Shaking(S_{i})$ \Comment{Aims to obtain an alternative searching point}
        \State $p \gets p+1$
        \EndIf
        \State $S^{'}_{i+1}, S_{i+1} \gets S_{i}$ 
        \State $i \gets i+1$ 
    \EndWhile
\EndWhile
\State \textbf{return}  $S^{*}$ 
\end{algorithmic}
\end{algorithm}

\subsubsection{Initial solution}
Unlike TSPD where any permutation of the customers nodes is a feasible route, in EVTSPD-FF, finding an feasible initial solution is non-trivial, especially when the EV's energy constraints are tight. A simple initial solution of EVTSPD-FF is using the EV to serve all the customers and ignore the drone completely. However, it is still non-trivial to obtain a feasible EV route. In the past literature, there exist various ways of finding a feasible solution. Some examples are the Modified Clarke-Wright (MCWS) algorithm proposed in \citep{Erdogan2012}, dynamic programming approach introduced in BP algorithm and PseudoGreedy algorithm, a greedy method proposed in \citep{felipe2014heuristic}. In this research, the MCWS method is used to generate a feasible EV route in VNS due to its simplicity.

\subsubsection{Neighbourhood structure}
As a critical component of the VNS algorithm, the neighborhood structure is used to escape the local optima so that the algorithm can explore the new space in the feasible region. The performance and effectiveness of the neighborhood operator directly affect the algorithm's overall performance. In this research, we use a variety of neighborhood operators whose performance has been tested in the past literature. All the used operators used in this research are shown in Table \ref{tab:VNS_operator}:

\begin{table}[H]
\begin{center}
\caption{A list of key neighbourhood structures used in VNS} \label{tab:VNS_operator}
\begin{tabular}{ m{2.5cm} m{9cm} m{4cm} } 
\toprule
\textbf{Name} & \textbf{Description}  & \textbf{Reference} \\
\hline
MakeFly & Let drone serve a customer based on saving  & \citep{Agatz2018}  \\
PushLeft & Move launch node of  drone sortie to left & \citep{Agatz2018}  \\
PushRight & Move retrieve node of drone sortie to right & \citep{Agatz2018} \\
TwoOpt & Choose two arcs in truck's route and reconnect them & \citep{de2020variable}  \\
Exchange 1.1 & Swap location of two nodes in EV's route & \citep{de2020variable}  \\
Exchange 2.1 & Swap location of two nodes with another node & \citep{de2020variable}  \\
Exchange 2.2 & Swap location of two nodes with other two nodes & \citep{de2020variable}  \\
RelocateCustomer & Let drone serve as many customers as possible & \citep{de2020variable} \\
Exchange 3.1 & Swap location of three consecutive nodes with another node & \citep{campuzano2021multi}  \\
Exchange 3.2 & Swap location of three consecutive nodes with other two node & \citep{campuzano2021multi}  \\
ReinsertCS & Delete one random CS node from EV's route and insert another one & This paper  \\
\bottomrule
\end{tabular}
\end{center}
\end{table}

\subsubsection{Maintaining solution feasibility}
As mentioned previously, the major complication of EVTSPD-FF compared to its TSPD counterpart is that it is difficult for EVTSPD-FF to maintain the solution's feasibility during the search process. In VNS, a saving-based greedy method is used to insert charging stations into found solutions so that its feasibility can be maintained.  \par

Note that maintaining solution feasibility is a costly process in VNS. Checking the solution's feasibility is of computational complexity $\mathcal{O}(C_{t})$ where $C_{t}$ is the number of nodes in the current EV's route, while inserting one CS node into the current solution is of computational complexity $\mathcal{O}(C_{t}|S|)$ where $|S|$ is the number of CS nodes in the network. Meanwhile, most of the neighborhood structure operator requires at most quadratic time $\mathcal{O}(C_{t}^{2})$. This indicates that, compared to TSPD, for each new solution found in the searching process, the number of operations needed to maintain the solution's feasibility is at least doubled. In practice, the increased complexity is even more significant because the inserting CS function might be called several times to obtain a feasible solution.

\section{Numerical Analysis}
In this section, the effectiveness of the proposed MILP formulation, the efficiency of the exact BP method and the performance of the VNS heuristic are all compared and tested. Then, a real world case study is conducted on the downtown Austin network to demonstrate the practicality of the proposed model.

\subsection{Experimental setting and implementation configuration}
While benchmark instances are available for drone related routing problems \citep{Agatz2018}, without the location of charging station nodes, these instances are not guaranteed to be feasible in the case of EVTSPD-FF with randomly generated charging stations. Thus, in this paper, we conduct a numerical analysis on randomly generated instances with realistic parameter setting. \par

In this test, for all the generated instances, the single depot is located at $(0, 0)$, and the coordinates of the customer nodes and the charging station nodes are uniformly distributed between -20 km and 20 km. The number of customers $|C|$ is an input parameter while the number of charging station nodes is roughly a third of $|C|$. The distance matrix for the electric vehicle is calculated using the Manhattan metric, while the Euclidean metric is used for the UAV to reflect its greater mobility options. For each randomly generated instance, the MCWS algorithm is used to solve the instance as electrical vehicle TSP to guarantee the instance is feasible for EVTSPD-FF. \par

In this test, we assume the electric van is Alke model ATX340E and the drone is DJI MATRICE 600 PRO. The parameter setting associated with the electric van and drone are mostly adopted from the company's official website \citep{AlkeVan, DJI}. As per the Federal Aviation Administration's regulation, the maximum allowed altitude is 400 feet (122 meters) above ground level. In this research, the drone's cruise altitude is estimated to be 50 meters. The drone's launch time is estimated to be the cruise altitude divided by the taking off vertical speed. In the table, the launch time also includes the time need for swapping battery and reloading parcel, which indicate that the launch time is greater than the retrieve time.

\begin{table}[H]
\begin{center}
\caption{A list of key parameters used in tests} \label{tab:test_para}
\begin{tabular}{ m{10cm} m{3cm} } 
\toprule
\textbf{Parameter}  & \textbf{Value} \\
\hline
\textbf{For Network:} &  \\
Number of customers $|C|$ & dedicated by user  \\
Number of charging stations $|S|$ & $\geq 3$ \\
\\
\textbf{For EV (Alke TX340E with 10kWh battery):} & \\
Driving speed & 40 km/h \\
Maximum driving range (time) & 100 km (2.5 h) \\
Charging time at CS node & 2 min \\
\\
\textbf{For drone (DJI MATRICE 600 PRO):} & \\
Flying speed & 60 km/h \\
Maximum flight range (time) & 20 km (20 minutes) \\
Drone-EV energy consumption rate ratio & 0.4 \\
Payload capacity & 6 kg \\
Launch time & 100 s \\
Retrieve time & 20 s \\
\bottomrule
\end{tabular}
\end{center}
\end{table}

The proposed branch-and-price algorithm and ALNS are coded in Python, while the MILP formulation is implemented in Pyomo and solved with ILOG's CPLEX Concert Technology solver (version 12.6.3). All experiments are run on a 3.6 GHz Intel Core i7 desktop with 32 GB RAM.\par

\subsection{Comparison between MILP formulation and BCP algorithm}
We first compare the performance between the two exact solution methods: solving the MILP formulation with commercial solver (CPLEX) and the proposed BP method. The computational results are shown in Table \ref{tab:VNS_performance1}, which is a summary for the average test results for instances with different size ($|C| = 5,6,7,8,9,10$) and different drone-truck travel speed ratio ($\alpha = 1.5, 2, 3$). In the table, Column \textit{No.ins} reports the number of instances tested. \textit{Arc} represents the result from the arc-based MILP model. As mentioned before, BP with three different ng-set sizes is implemented. In the column \textit{Opt}, the value indicates the number of instances solved by different approaches. Note that the exact optimal solution is not available for large instances ($i \geq 8$). In this case, the solution with the lowest total cost is considered the optimal solution among all the returned solutions of all the four solution methods. Column \textit{RunningTime} reveals the average operational time of different approaches over the solved instances, which is measured in seconds. Column \textit{BB nodes} presents the number of branch-and-bound nodes searched to obtain the optimal solution. The last column, \textit{GapRoot} reveals the average gap at the root node between the optimal solution cost and the lower bound returned by the root node. Note that only the instance that is solved by the corresponding approach is considered for each approach. \par

\begin{landscape}
\begin{table}[H]
\begin{center}
\caption{Performance comparison between MILP and BP} \label{tab:VNS_performance1}
\begin{tabular}{m{0.4cm} m{0.4cm} m{0.4cm} m{0.6cm} m{0.9cm} m{0.9cm} m{0.9cm} m{0.9cm} m{0.9cm} m{0.9cm} m{0.9cm} m{0.9cm} m{0.9cm} m{0.9cm}  m{0.9cm} m{0.9cm} m{0.9cm} m{0.9cm}}
\toprule
\multicolumn{1}{l}{} & \multicolumn{1}{l}{} & \multicolumn{1}{l}{} & \multicolumn{1}{l}{} & \multicolumn{4}{c}{Opt} & \multicolumn{4}{c}{RunningTime} & \multicolumn{3}{c}{BBNodes} & \multicolumn{3}{c}{GapRoot(\%)} \\
\cline{5-8}   \cline{9-12}    \cline{13-15}  \cline{16-18} \\
$|C|$ & $|S|$ & $\alpha$ & No.ins  & Arc & BP-5 & BP-6 & BP-7  & Arc & BP-5 & BP-6 & BP-7 & BP-5 & BP-6 & BP-7 & BP-5 & BP-6 & BP-7 \\
\hline
5 & 3 & 1.5 & 20 & 20 & 20 & N/A & N/A  & 0.50 & 0.40 & N/A & N/A & 1.00 & N/A & N/A & 0.00 & N/A & N/A \\
 & 3 & 2 & 20 & 20 & 20  & N/A & N/A  & 0.50 & 0.40 & N/A & N/A & 1.00 & N/A & N/A & 0.00 & N/A & N/A \\
 & 3 & 3 & 20 & 20 & 20  & N/A & N/A  & 0.60 & 0.40 & N/A & N/A & 1.00 & N/A & N/A & 0.00 & N/A & N/A \\
6 & 3 & 1.5 & 20 & 20  & 20 & 20 & N/A  & 7.80 & 11.30 & 2.20 & N/A & 1.40 & 1.00 & N/A & 1.40 & 0.00 & N/A \\
 & 3 & 2 & 20 & 20 &  20 & 20 & N/A  & 7.70 & 11.50 & 2.30 & N/A & 1.40 & 1.00 & N/A & 1.40 & 0.00 & N/A \\
 & 3 & 3 & 20 & 20 &  20 & 20 & N/A  & 7.70 & 11.40 & 2.40 & N/A & 1.40 & 1.00 & N/A & 1.40 & 0.00 & N/A \\
7 & 3 & 1.5 & 20  & 20 & 18 & 18 & 19  & 81.00 & 96.60 & 148.40 & 18.00 & 2.60 & 1.50 & 1.00 & 2.50 & 1.33 & 0.00 \\
 & 3 & 2 & 20  & 20 & 18 & 18 & 19  & 80.00 & 96.00 & 148.10 & 18.30 & 2.60 & 1.50 & 1.00 & 2.50 & 1.33 & 0.00 \\
 & 3 & 3 & 20  & 20 & 18 & 18 & 19  & 82.00 & 98.30 & 148.60 & 18.20 & 2.60 & 1.50 & 1.00 & 2.50 & 1.33 & 0.00 \\
8 & 3 & 1.5 & 20  & 20 & 12 & 12 & 13 & 540.75 & 8.63 & 1.73 & 421.73 & 3.23 & 1.70 & 1.10 & 3.18 & 1.51 & 0.34 \\
 & 3 & 2 & 20  & 20 & 12 & 12 & 13  & 540.75 & 8.55 & 1.80 & 423.00 & 3.23 & 1.70 & 1.10 & 3.18 & 1.51 & 0.37 \\
 & 3 & 3 & 20  & 20 & 12 & 12 & 13  & 540.75 & 72.45 & 111.30 & 13.50 & 2.90 & 1.70 & 1.10 & 3.18 & 1.51 & 0.37 \\
9 & 3 & 1.5 & 10  & N/A & 5 & 5 & 5  & >3600 & 887.46 & 1103.00 & 2407.00 & 2.50 & 4.14 & 2.00 & 2.68 & 4.79 & 0.82 \\
 & 3 & 2 & 10  & N/A & 5 & 5 & 5  & >3600 & 887.00 & 1123.46 & 2464.15 & 3.00 & 4.14 & 2.00 & 2.68 & 4.08 & 0.78 \\
 & 3 & 3 & 10  & N/A & 5 & 5 & 5  & >3600 & 856.00 & 1135.56 & 2469.47 & 3.00 & 4.14 & 2.00 & 2.68 & 4.07 & 0.52 \\
10 & 3 & 1.5 & 10  & N/A & 3 & 3 & 3  & >3600 & 2122.00 & 2891.00 & 3249.00 & 9.00 & 8.60 & 8.00 & 25.76 & 15.73 & 9.36 \\
 & 3 & 2 & 10  & N/A & 3 & 3 & 3 & >3600 & 2140.60 & 2846.20 & 3248.50 & 9.00 & 8.60 & 8.00 & 25.76 & 15.73 & 9.36 \\
 & 3 & 3 & 10  & N/A & 3 & 3 & 3 & >3600 & 2110.40 & 2865.10 & 3200.40 & 9.00 & 8.60 & 8.00 & 25.76 & 15.73 & 9.36 \\
\bottomrule
\end{tabular}
\end{center}
\end{table}
\end{landscape}

As can be seen from Table \ref{tab:VNS_performance1}, it is obvious that the branch-and-price algorithm is more efficient than the proposed MILP models defined on the multigraph, regardless of the value of the parameter $\alpha$. Specifically, the arc-based model can solve instances containing 8 customers while the BP method can solve instances containing up to 10 customers in one hour. \par

As we mentioned earlier, BP-5 is the most efficient in solving the pricing problem among all three BP methods, while BP-7 can obtain the best lower bounding at the root node. This property is also demonstrated by the results in Table \ref{tab:VNS_performance1}. According to the table, BP-7 solves the most cases among the three. This is mainly because BP-7 needs to examine fewer nodes before termination in the branch-and-bound procedure by obtaining a better lower bound at the root node. This result is also illustrated in the "\textit{BB nodes}" column, as the average BB nodes for BP-7 is much lower than the other two methods. \par

Besides, comparing the performance of the three BP methods, it seems that when $i \leq 7$, BP-$i$ is the most efficient method and BP-5 is more efficient than the other two methods when $i \geq 8$. In the first case, when $i \leq 7$, the branch-and-bound tree for method BP-$i$ only contains the root node, which indicates that the problem could be solved in one iteration by solving the root node. However, when $i \geq 8$, all three BP methods need to examine other branch-and-bound nodes other than the root node to terminate the algorithm. As the results suggest, as the number of customers increases, the average BB nodes in the BP method increase accordingly, from less than 3 when $i \leq 7$ to around 8 when $i = 10$. Meanwhile, when $i \geq 8$, the average computational time for BP-5 is much lower than the other two methods, as it is simpler to solve the pricing problem for BP-5 than BP-6 and BP-7. \par

Furthermore, the test results with various drone-truck travel speed ratio values indicate that the value of parameter $\alpha$ has little impact on the overall computational efficiency of all three solution methods. This is likely because the value of $\alpha$ only affects the number of feasible drone sorties in the instances whose impact on the total number of decision variables is minimal. 

Besides, we also conduct additional tests on the performance of MILP model and BP5 method on some commonly seen variants of EVTSPD-FF. The following EVTSPD-FF variants are considered:
\begin{itemize}
    \item \textit{Base}: the base case of EVTSPD-FF without any side constraints
    \item \textit{LRT}: the EVTSPD-FF variant considers the launch/retrieve time, whose value are shown in Table \ref{tab:test_para}
    \item \textit{Range}: the EVTSPD-FF variant which assumes the drone's flight range is dependent on the weight of parcel it carries. In this research, we adopts a simplified assumption that its flight range is doubled when no parcel is on boarded and that the range decreases linearly with the weight of the parcel until it reaches the base range value with full-capacity payload.
    \item \textit{Loop}: the EVTSPD-FF variant which allows the self-loop
\end{itemize}

The computational results are shown in Table \ref{tab:VNS_performance2}, where the unit for each value is second. It can been seen that the BP method is very reliable in handling additional side constraints as it is fairly simple to consider these extra side constraints in dynamic programming propagation. For the MILP model, the computational time of maximum leg variant and self-loop variant is about 35\% higher than normal EVTSPD-FF. 

\begin{table}[H]
\begin{center}
\caption{Performance comparison between MILP and BP} \label{tab:VNS_performance2}
\begin{tabular}{m{0.4cm} m{0.4cm} m{0.4cm} m{0.6cm} m{0.9cm} m{0.9cm} m{0.9cm} m{0.9cm} m{0.9cm} m{0.9cm} m{0.9cm} m{0.9cm} m{0.9cm} m{0.9cm}}
\toprule
\multicolumn{1}{l}{} &
  \multicolumn{1}{l}{} &
  \multicolumn{1}{l}{} &
  \multicolumn{1}{l}{} &
  \multicolumn{2}{c}{\textbf{Base}} &
  \multicolumn{2}{c}{\textbf{LRT}} &
  \multicolumn{2}{c}{\textbf{Range}} &
  \multicolumn{2}{c}{\textbf{MaxLeg}} &
  \multicolumn{2}{c}{\textbf{Loop}} \\
\cline{5-6}   \cline{7-8}    \cline{9-10}  \cline{11-12} \cline{13-14}\\
$|C|$ & $|S|$ & $\alpha$ & No.ins & Arc   & BP5    & Arc   & BP5    & Arc   & BP5    & Arc   & BP5    & Arc   & BP5\\
\hline 
5 & 3    & 1.5    & 10     & 0.80  & 0.70   & 0.80  & 0.70   & 0.80  & 0.70   & 1.20  & 0.70   & 1.3   & 0.70    \\
  & 3    & 2      & 10     & 0.80  & 0.70   & 0.80  & 0.70   & 0.80  & 0.70   & 1.20  & 0.70   & 1.3   & 0.70    \\
  & 3    & 3      & 10     & 0.80  & 0.60   & 0.80  & 0.60   & 0.80  & 0.70   & 1.20  & 0.70   & 1.3   & 0.80    \\
6 & 3    & 1.5    & 10     & 7.90  & 11.10  & 7.90  & 11.10  & 7.90  & 11.10  & 10.68 & 11.10  & 11.2  & 11.70   \\
  & 3    & 2      & 10     & 7.90  & 12.10  & 7.90  & 11.00  & 7.90  & 10.90  & 10.68 & 11.50  & 11.2  & 12.10   \\
  & 3    & 3      & 10     & 7.90  & 11.00  & 7.90  & 11.10  & 7.90  & 11.10  & 10.68 & 11.60  & 11.2  & 12.10   \\
7 & 3    & 1.5    & 10     & 70.8  & 95.4   & 70.8  & 96.4   & 70.8  & 96.4   & 100.5 & 95.4   & 105.3 & 94.4    \\
  & 3    & 2      & 10     & 70.8  & 97.4   & 70.8  & 97.1   & 70.8  & 95.9   & 100.5 & 96.5   & 105.3 & 93.9    \\
  & 3    & 3      & 10     & 70.8  & 97.8   & 70.8  & 98.9   & 70.8  & 96.1   & 100.5 & 96.9   & 105.3 & 95.4    \\
8 & 3    & 1.5    & 10     & 634.3 & 433.4  & 634.3 & 433.4  & 634.3 & 433.4  & 856.1 & 435.5  & 891.7 & 430.20  \\
  & 3    & 2      & 10     & 634.3 & 433.9  & 634.3 & 432.4  & 634.3 & 433.4  & 856.1 & 437.4  & 891.7 & 430.10  \\
  & 3    & 3      & 10     & 634.3 & 435.4  & 634.3 & 436.4  & 634.3 & 433.4  & 856.1 & 436.4  & 891.7 & 430.50  \\
9 & 3    & 1.5    & 10     & >3600 & 1849.5 & >3600 & 1849.5 & >3600 & 1849.5 & >3600 & 1757.5 & >3600 & 1768.00 \\
  & 3    & 2      & 10     & >3600 & 1816.5 & >3600 & 1875.5 & >3600 & 1868.5 & >3600 & 1757.5 & >3600 & 1778.00 \\
  & 3    & 3      & 10     & >3600 & 1829.5 & >3600 & 1849.5 & >3600 & 1871.5 & >3600 & 1757.5 & >3600 & 1798.00 \\
\bottomrule
\end{tabular}
\end{center}
\end{table}

\newpage
\subsection{Performance evaluation of VNS}
In this section, we analyze the performance of the proposed variable neighourhood search heuristic and compare it with the exact branch-and-cut method. The parameter setting of the tested instances are the same as previously described. The test results are shown in Table \ref{tab:VNS_performance}. In the table, the unit for the final cost columns is 10 seconds and the optimality gap of VNS is calculated as the difference between final results of two methods divided by the optimal cost.\par

As can be seen from the table, the performance of VNS is reasonably acceptable as the optimality gap is less than 3.5\% for instances containing less than 8 customers and 3 charging stations while the running time of VNS is significantly less than that of BP. For instances with more than 10 customers, the BP fails to terminate in 1 hour while the VNS could solve instances with 25 customers in about 1 minute. Although we expect the optimality gap of VNS to grow as the size of the instances increases, the current test results indicate the VNS is competitive to exact solution methods.

\begin{table}[H]
\begin{center}
\caption{Performance comparison between VNS and BP} \label{tab:VNS_performance}
\begin{tabular}{m{0.5cm} m{0.5cm} m{0.5cm} m{0.6cm} m{1cm} m{1cm} m{1cm} m{1cm} m{1cm} m{1cm}}
\toprule
  & & & & 
\multicolumn{2}{c}{\textbf{Final Cost}} & 
\multicolumn{2}{c}{\textbf{Running Time (s)}} & 
\multicolumn{2}{c}{\textbf{Gap (\%)}} \\
\cline{5-6}   \cline{7-8}    \cline{9-10} \\
$|C|$ & $|S|$ & $\alpha$ & No.ins & BP   & VNS    & BP   & VNS   & BP   & VNS\\
\hline 
5 & 3 & 1.5 & 20 & 990.00  & 990.80  & 3.10   & 0.60  & 0.00 & 0.08 \\
  & 3 & 2   & 20 & 905.20  & 905.20  & 3.00   & 0.60  & 0.00 & 0.00 \\
  & 3 & 2.5 & 20 & 765.50  & 765.50  & 3.00   & 0.60  & 0.00 & 0.00 \\
6  & 3 & 1.5 & 20 & 1089.40 & 1108.50 & 36.20  & 1.10   & 0.00 & 1.75 \\
   & 3 & 2   & 20 & 1015.80 & 1020.10 & 36.10  & 1.10  & 0.00 & 0.42 \\
   & 3 & 2.5 & 20 & 857.70  & 865.70  & 37.20  & 1.10  & 0.00 & 0.93 \\
7  & 3 & 1.5 & 20 & 1172.00  & 1212.00  & 365.00 & 2.00 & 0.00 & 3.41 \\
    & 3 & 2   & 20 & 1010.50 & 1042.20 & 366.30 & 1.94  & 0.00 & 3.14 \\
    & 3 & 2.5 & 20 & 794.50  & 812.20  & 367.50 & 1.92  & 0.00 & 2.23 \\
10  & 4 & 1.5 & 20 & N/A     & 1453.80 & N/A    & 5.86  & N/A  & N/A \\
    & 4 & 2   & 20 & N/A     & 1364.40 & N/A    & 5.74  & N/A  & N/A \\
    & 4 & 2.5 & 20 & N/A     & 1258.20 & N/A    & 5.68  & N/A  & N/A   \\
15  & 6 & 1.5 & 10 & N/A     & 1694.80 & N/A    & 18.42 & N/A  & N/A   \\
    & 6 & 2   & 10 & N/A     & 1637.80 & N/A    & 17.99 & N/A  & N/A   \\
    & 6 & 2.5 & 10 & N/A     & 1584.40 & N/A    & 17.84 & N/A  & N/A   \\
20  & 8 & 1.5 & 10 & N/A     & 1851.20 & N/A    & 35.86 & N/A  & N/A   \\
    & 8 & 2   & 10 & N/A     & 1746.40 & N/A    & 35.01 & N/A  & N/A   \\
    & 8 & 2.5 & 10 & N/A     & 1636.80 & N/A    & 35.04 & N/A  & N/A  \\
25 & 8 & 1.5 & 10 & N/A     & 2168.75 & N/A    & 68.70 & N/A  & N/A    \\
   & 8 & 2   & 10 & N/A     & 2015.34  & N/A    & 67.90 & N/A  & N/A   \\
   & 8 & 2.5 & 10 & N/A     & 1954.76 & N/A    & 69.20 & N/A  & N/A    \\          
\bottomrule
\end{tabular}
\end{center}
\end{table}

\newpage
\subsection{Real-world case study}
In this section, a real-world case study is conducted to illustrate the effectiveness of the proposed algorithm in solving EVTSPD-FF with practical size. This study also analyzes the effects of several key parameters on the final delivery time. In this case study, the downtown Austin network is analyzed. A total of 35 nodes are chosen randomly from the network, most of which are the centroids of ZIP Code Tabulated Areas (ZCTAs) in Austin metro area, while ensuring that the resulting EVTSPD-FF problem is feasible. Out of these 35 nodes, 10 nodes are selected as the charging stations and the remaining 25 nodes are considered as customer nodes or depot. Two different layouts are test. In the first layout we pick a node that lies in the central area of the network as the depot while in the second layout the depot is chosen as a node that lies in the suburban district of Austin. These two layouts are named as central layout and side layout, respectively. Thus, the final network contains one depot, 24 customers and 10 charging stations. All customers can be served by either the EV or the UAV. The travel time of the EV from one node to another node is estimated using the Google map Python API for the peak hour of a typical Monday. Thus, the resulting travel time matrix is asymmetric. Besides, the travel time of the UAV from one node to another is calculated as the direct distance between the two nodes divided by the drone's speed, which is 60 km/h. All the other parameters are the same as the previously described. A representation of the network is shown in Figure \ref{fig:AustinNet}. \par

With the default setting, the completion delivery time obtained via the proposed VNS heuristic is 16040 seconds (approximately 4.5 hours), and the final route consists of 14 customers being served by the UAV and ten customers being served by the EV. \par   

\begin{figure}[H]
    \centering
    \includegraphics[width = 0.8\textwidth]{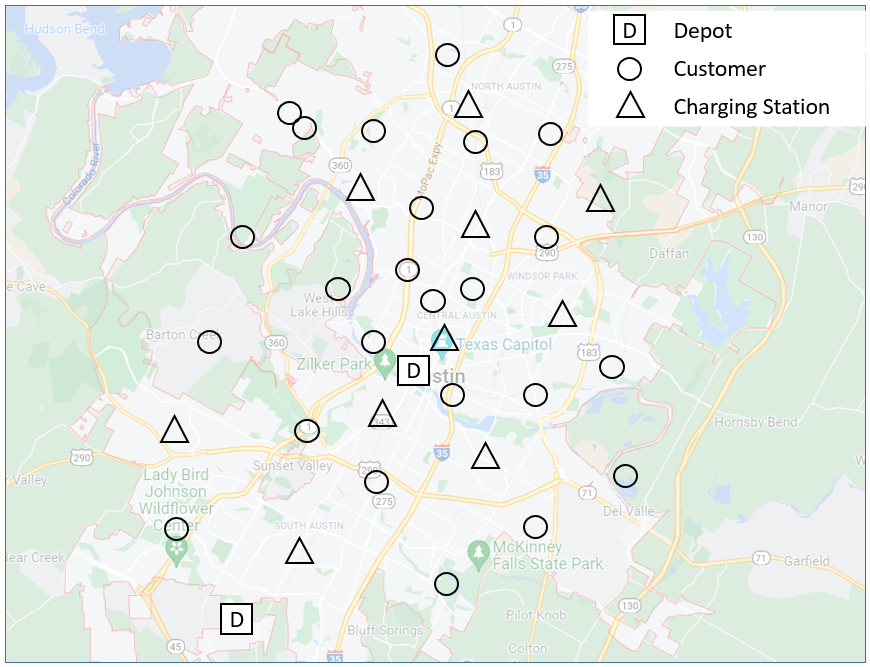}
    \caption{A representation of the Austin network}
    \label{fig:AustinNet}
\end{figure}

In the rest of this section, we study the impact of characteristics of different parameters on the performance of this EV-UAV delivery system, from which we will gain insights into this new mode of transportation. All the results are the best-found solution based on ten independent runs of VNS. 

\subsubsection{UAV speed}
The UAV is typically faster than the EV as it can take shortcuts without being affected by ground traffic congestion. The X-Wing project initiated by Google declares their delivery UAVs can reach the speed of 120 km/h when tested in a suburban area in Australia. This section explores the effect of the drone's travelling speed on the final delivery cost. Note that the default travel speed is 60km/h and three additional speed are tested, namely, 50 km/h, 70 km/h and 80 km/h. The final solution cost of different drone travelling speed is shown in Figure \ref{fig:drone_speed}. 

As shown, the delivery completion time decreases as the travelling speed of the UAV increases, for both layouts. As the UAV speed increases, the total delivery time, the EV travel time, and the total waiting time reduce gradually. For central layout, the final route completion cost decreases about 16.67\% when drone's speed increases from 50 km/h to 80 km/h. For side layout this value about 11.1\%. This result indicates that increasing the UAV speed is an effective approach to reduce the risk and operational cost incurred by the waiting period during the delivery. 

\begin{figure}[H]
    \centering
    \includegraphics[width = 0.8\textwidth]{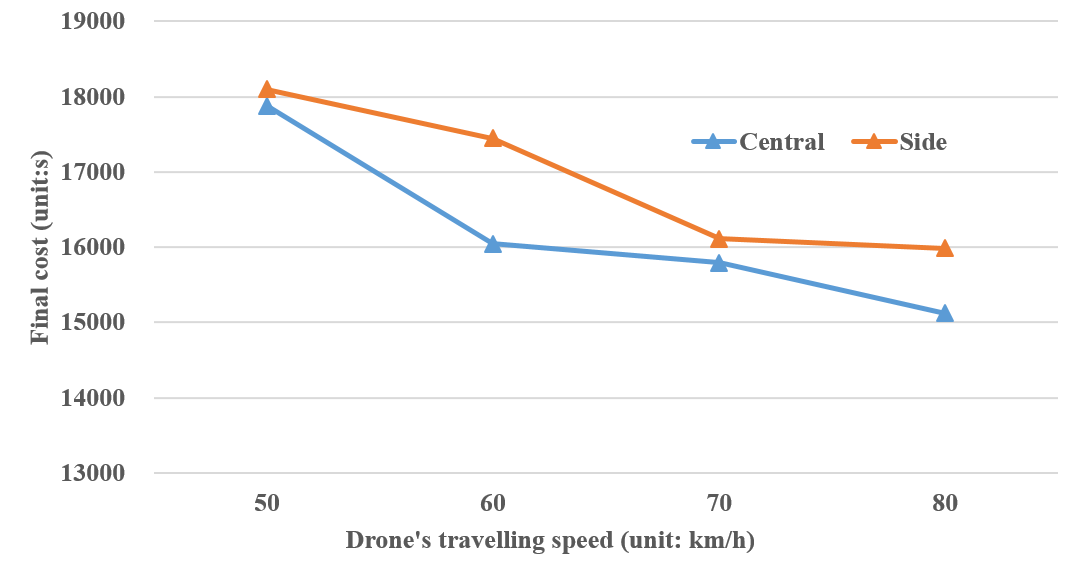}
    \caption{Sensitivity analysis of drone's speed}
    \label{fig:drone_speed}
\end{figure}

\subsubsection{EV's driving range}
The driving range of an EV depends on the size of the battery it carries. However, as battery size increases, the operational and maintenance cost also grows, so logistics companies deploying electric vehicles must make a trade-off between these two factors. In EVTSPD-FF, the driving range of the EV affects the final delivery completion time in that the low driving range indicates EV has to visit charging stations frequently. Tighter energy constraints also make it difficult for the heuristic algorithm to find a "good" feasible solution. Denote the EV's driving range as $Q^{T}$, which is measured in seconds. This section presents the result of how EV's driving range affects the solution route cost, as illustrated in Figure \ref{fig:EV_range}. As presented, for side layout the final delivery time decreases with the EV driving range increases. However, this trend is not obvious for central layout. \par 

\begin{figure}[H]
    \centering
    \includegraphics[width = 0.8\textwidth]{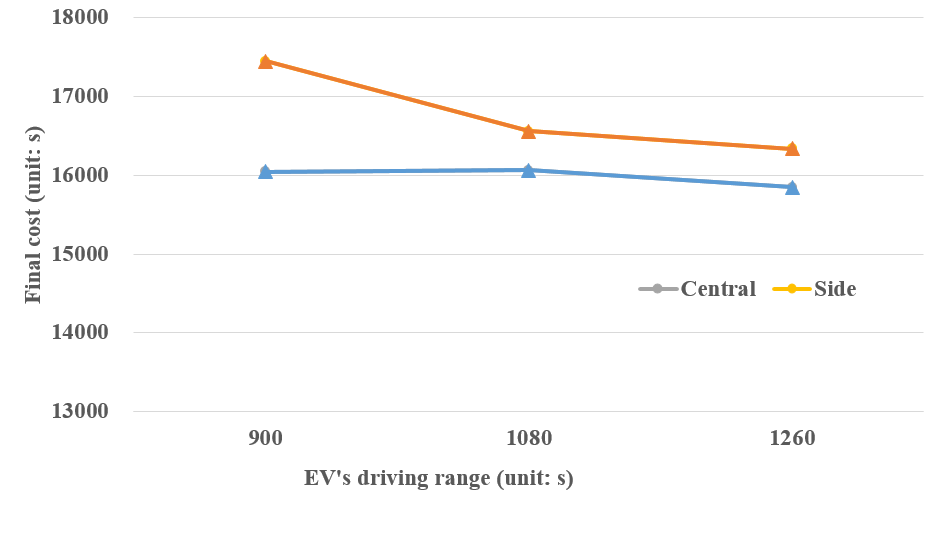}
    \caption{Sensitivity analysis of EV's driving range}
    \label{fig:EV_range}
\end{figure}

\section{Conclusion and future work}
This paper introduces the electric vehicle traveling salesman problem with the drone, a new delivery concept combining electric vehicle and drone. In EVTSPD, the electric vehicle and drone operate in a cooperative way,  where the EV serves as a drone hub that launches and retrieves the drone similar to the FSTSP introduced in \citep{Murray2015}  and TSPD introduced in \citep{Agatz2018}. Meanwhile, the unique features of the EVTSPD, compared to the past research, are two-fold. First, as the electric vehicle is modeled in the problem, the truck needs to visit charging station nodes to avoid depleting its energy. Second, we propose a novel shared energy assumption, which indicates that the EV and drone shared their energy in the operation process. EVTSPD aims to find a coordinated EV-drone route that minimizes total travel time and serves a set of customers while incorporating stops at charging stations in route plans to ensure sufficient charge.\par 

In this paper, the EV is assumed to be fully-charged with fixed time at charging station nodes. We present a novel arc-based mixed integer/linear programming formulation for EVTSPD-FF defined in a constructed multigraph. In this multigraph, all the feasible route between two non-charging station nodes are pre-calculated and stored in the network as an arc between the two nodes. In this way, the dimension of the multigraph is significantly smaller than the original network which needs to be augmented to enable multiple visit to one single charging station. Later in this paper we show that computational efficiency could be achieved by utilizing this multigraph and an exact branch-and-price algorithm is proposed to solve the EVTSPD-FF. In this BP algorithm, a dynamic programming method is used to propagate in the multigraph to obtain a feasible ng-route of the problem. This research also presents a variable neighbourhood search heuristic to solve the problem. \par

In the numerical analysis section, the computational efficiency of the MILP formulation, exact BP algorithm, and the proposed VNS heuristic are tested and compared. The test results indicate that the BP method is much more efficient than solving the MILP model via a commercial solver. The VNS is the fastest of all three method with reasonable optimality gap. Considering the enormous amount of constraints in the problem, the exact method can only be used to solve small-sized instances and for instances containing more than ten customers the VNS is the only viable method. Besides, a real-world case study based on downtown Austin network is performed to show that the VNS could solve the problem with practical size within a reasonable computational time. It also conducts several sensitivity analyses to illustrate how key parameters affect the final integrated route. The test results demonstrate that the UAV speed has a more significant influence on the final delivery time compared to the EV driving range. 

The EVTSPD-FF formulations, along with these solution techniques, will aid organizations with EV fleets in overcoming difficulties resulting from limited recharging infrastructure. The new delivery concept of using UAV and EV to perform last-mile delivery would result in financial and environmental benefits when considering the reduced operation cost of fueling and switching to drone, which does not require a costly human pilot. Besides, the research in the newly-merged approach will provide intuition for the future development of a more sophisticated delivery service. \par

There remain several practical challenges for drone delivery regarding payload capacity, safety, and public acceptance \citep{watkins2020ten}.  It remains to be seen whether battery-swapping stations (like the type we assume) or fast-charging stations ultimately make more economic sense for logistics fleets.  As the number of customers increases, it will become essential to consider a multi-vehicle version of the current problem, perhaps with heterogeneous EV and drone range and capacity or alternative (non full-charge) policies. All of these should be addressed in future research.

\section{Acknowledgements}
This research is based on work supported by the National Science Foundation under Grant No. 1826230, 1562109/1562291, 1562109, 1826337, 1636154, and 1254921. This work is also supported by the Center for Advanced Multimodal Mobility Solutions and Education (CAMMSE).

\bibliography{library}
\end{document}